\documentclass[11pt]{amsart}
\usepackage{amssymb}
\usepackage[hyphens]{url} 
\usepackage[dvips]{graphicx}
\usepackage{amsmath}
\usepackage{cite}
\usepackage{mathrsfs}
\usepackage[title]{appendix} 
\usepackage[utf8]{inputenc}
\usepackage[english]{babel}
\usepackage{hyperref}
\usepackage{cleveref}
\RequirePackage{textgreek}
\usepackage[all]{xy}
\usepackage{enumitem}
\usepackage{bm}
\usepackage{nicefrac}
\usepackage{mathabx}
\usepackage{tikz-cd}
\makeatother
\usepackage{babel}
\providecommand{\definitionname}{Definition}

\providecommand{\lemmaname}{Lemma}
\providecommand{\questionname}{Question}
\providecommand{\theoremname}{Theorem}
\numberwithin{equation}{section}

\hypersetup{
    colorlinks=true,
    linkcolor=blue,
    citecolor=cyan,
    urlcolor=red
}

\hypersetup{linktoc=page}
\newtheorem{thm}{\protect\theoremname}[section]
\theoremstyle{plain}
\newtheorem{lem}[thm]{\protect\lemmaname}
\newtheorem{prop}[thm]{Proposition}
\newtheorem{cor}[thm]{Corollary}

\theoremstyle{definition}
\newtheorem{defn}[thm]{\protect\definitionname}

\newtheorem{exm}[thm]{Example}

\newtheorem{rem}[thm]{Remark}
\newtheorem{fct}[thm]{Fact}
\newtheoremstyle{claimstyle}{}{}{}{}{}{}{ }{\textbf{Claim:\thmnote{ \textnormal{(#3)}}}}

\theoremstyle{claimstyle}

\author{Nachi Avraham-Re'em}
\address{Department of Mathematics, Chalmers University, G\"{o}teborg, Sweden}
\email{nachi.avraham@gmail.com \(\textstyle{or}\) nachman@chalmers.se}

\author{George Peterzil}
\address{Einstein Institute of Mathematics, The Hebrew University, Israel}
\email{george.peterzil@mail.huji.ac.il}

\thanks{The research was supported by ISF (grant No. 1180/22) and by the Knut and Alice Wallenberg Foundation (KAW 2021.0258).}

\date{\today}

\title[Uncountable Hyperfiniteness]{Uncountable Hyperfiniteness and\\ The Random Ratio Ergodic Theorem}



\begin{document}

\maketitle

\begin{abstract}
We show that the orbit equivalence relation of a free action of a locally compact group is hyperfinite (\`{a} la Connes--Feldman--Weiss) precisely when it is \emph{hypercompact}. This implies an uncountable version of the Ornstein--Weiss Theorem and that every locally compact group admitting a hypercompact probability preserving free action is amenable. We also establish an uncountable version of Danilenko's Random Ratio Ergodic Theorem. From this we deduce the \emph{Hopf dichotomy} for many nonsingular Bernoulli actions.
\end{abstract}

\tableofcontents

\section{Introduction}

A well-studied property of countable equivalence relations is the property of being \emph{hyperfinite}, which means that the equivalence relation is the increasing union of countably many \emph{finite} equivalence relations, where the finiteness is referring to the cardinality of the classes. Of a particular interest are Borel actions of countable groups whose associated \emph{orbit equivalence relation} (henceforth \emph{OER}) is hyperfinite. One of the milestones in this theory is the Ornstein--Weiss Theorem \cite[Theorem 6]{ornstein1980}, by which OERs arising from nonsingular free actions of amenable groups are \emph{measure-hyperfinite}, i.e. hyperfinite up to a zero measure set. The celebrated Connes--Feldman--Weiss Theorem \cite{connes1981} generalized this to amenable equivalence relations (not necessarily OERs).

The usual notion of hyperfiniteness applies to countable equivalence relations only, thus the aforementioned theorems deal with countable groups. In dealing with general locally compact second countable (henceforth \emph{lcsc}) groups, Connes, Feldman \& Weiss used cross sections \cite[p. 447]{connes1981},\footnote{The terminology in the field has evolved in an inconsistent way. What is referred to by Connes, Feldman \& Weiss as \emph{transversal} has evolved to what is nowadays usually called \emph{cross section}, meaning a set which intersects every class countably many times, while \emph{transversal} nowadays refers to a set intersecting every class exactly once.} which will be introduced below in Section \ref{sct:cross}. Thus, they called an uncountable OER \emph{hyperfinite} if its restriction to one (hence every; see Proposition \ref{prop:kechris} below) cross section is hyperfinite. In order to avoid confusion, we will call this {\bf sectional-hyperfinite} (see Definition \ref{dfn:secthyperf} below). Thus, the Connes--Feldman--Weiss Theorem asserts that every OER of an lcsc amenable group is measure-sectional-hyperfinite \cite[Corollary 18]{connes1981}.

In this work we focus on OERs arising from Borel free actions of an lcsc group \(G\) on a standard Borel space \(X\). Such an OER will be denoted by
\[E_{G}^{X}.\]
While the classes of \(E_{G}^{X}\) are typically uncountable, using the freeness of the action there is a natural way to define hypercompactness by pushing forward the topology from the acting group to the orbits. A precise definition will be presented in Section \ref{sct:hypercomp}. With this natural concept defined, we have the following characterization which is purely in the Borel category:

\begin{thm}
\label{thm:CFW}
Let \(G\) be an lcsc group and \(X\) a free Borel \(G\)-space. Then \(E_{G}^{X}\) is sectional-hyperfinite if and only if it is hypercompact.
\end{thm}

When adding a measure, from the Connes--Feldman--Weiss Theorem we obtain an uncountable version of the Ornstein--Weiss Theorem:

\begin{cor}
\label{cor:OW}
Every nonsingular free action of an lcsc amenable group is measure-hypercompact.
\end{cor}

We continue to study the asymptotic invariance of hypercompact OERs, as presented in Section \ref{sct:asympinv}, and use it to get an uncountable version of a well-known fact in countable groups (see \cite[Proposition 4.3.3]{zimmer2013} and with the Ornstein--Weiss Theorem):

\begin{thm}
\label{thm:pmp}
An lcsc group \(G\) is amenable if and only if one (hence all) of its probability preserving free \(G\)-spaces is measure-hypercompact.
\end{thm}

Hyperfiniteness or hypercompactness of an action provides a natural way to take ergodic averages of a function and study their asymptotic behaviour. This was done by Danilenko \cite[Appendix A]{danilenko2019} in the countable case, where he established a \emph{Random Ratio Ergodic Theorem} for nonsingular actions of countable amenable groups using the Ornstein--Weiss Theorem. In the next we exploit Theorem \ref{thm:pmp} in order to establish an uncountable version of Danilenko's theorem. We will make use of a natural notion of \emph{Random F\o lner sequence} that will be presented in Section \ref{sct:tworret}.

\begin{thm}[Random Ratio Ergodic Theorem]
\label{thm:rret}
Let \(G\) be an lcsc amenable group and \(\left(X,\mu\right)\) a nonsingular probability \(G\)-space (not necessarily free). Then there exists a random F\o lner sequence \(\mathcal{S}_{1}\subseteq\mathcal{S}_{2}\subseteq\dotsm\) of \(G\) such that for every \(f\in L^{1}\left(X,\mu\right)\),
\[\lim_{n\to\infty}\frac{\int_{\mathcal{S}_{n}}\frac{d\mu\circ g}{d\mu}\left(x\right)f\left(g.x\right)d\lambda\left(g\right)}{\int_{\mathcal{S}_{n}}\frac{d\mu\circ g}{d\mu}\left(x\right)d\lambda\left(g\right)}=\mathbb{E}\left(f\mid\mathrm{Inv}_{G}\left(X\right)\right)\left(x\right)\]
for almost every realization \(\left(\mathcal{S}_{1},\mathcal{S}_{2},\dotsc\right)\), both \(\mu\)-a.e. and in \(L^{1}\left(X,\mu\right)\).
\end{thm}

As it was shown by Hochman \cite{hochman2013}, when fixing one realization of \(\mathcal{S}_{1}\subseteq\mathcal{S}_{2}\subseteq\dotsm\) it is not always true that the limit in Theorem \ref{thm:rret} holds for all functions in \(L^{1}\left(X,\mu\right)\) at once. Nevertheless, it was observed by Danilenko that one may restrict the attention to a countable family of functions in \(L^{1}\left(X,\mu\right)\), and obtain the following particularly useful corollary:

\begin{cor}
\label{cor:danilrret}
Let \(G\) be an lcsc amenable group and \(\left(X,\mu\right)\) a nonsingular probability \(G\)-space. Then for every countable family \(\mathcal{L}\subset L^{1}\left(X,\mu\right)\) there exists a F\o lner sequence \(S_{1}\subseteq S_{2}\subseteq\dotsm\) of \(G\) such that for every \(f\in\mathcal{L}\),
\[\lim_{n\to\infty}\frac{\int_{S_{n}}\frac{d\mu\circ g}{d\mu}\left(x\right)f\left(g.x\right)d\lambda\left(g\right)}{\int_{S_{n}}\frac{d\mu\circ g}{d\mu}\left(x\right)d\lambda\left(g\right)}=\mathbb{E}\left(f\mid\mathrm{Inv}_{G}\left(X\right)\right)\left(x\right)\]
both \(\mu\)-a.e. and in \(L^{1}\left(X,\mu\right)\).
\end{cor}

The Random Ratio Ergodic Theorem in countable groups proved itself useful in ergodic theory in recent years, particularly because it is not limited to probability preserving actions but applies also to nonsingular actions, where the pointwise ergodic theorem is generally unavailable. The classical \emph{Hopf method}, originally used by Hopf in the probability preserving category to prove the ergodicity of the geodesic flow, was developed in recent years by Kosloff \cite{kosloff2019} and then by Danilenko \cite{danilenko2019} in the nonsingular category. Originally, Kosloff suggested this method for \emph{Ratio Ergodic Theorem countable groups} (see \cite[\S3.3]{kosloff2019}), and Danilenko, observing that the Random Ratio Ergodic Theorem is sufficient, showed that this method applies to all countable amenable groups.

The works of Kosloff and Danilenko concentrate on establishing ergodicity in nonsingular Bernoulli and Markov shifts over countable groups. In Section \ref{sct:Hopf}, we will use Theorem \ref{thm:rret} to demonstrate that certain \emph{nonsingular Bernoulli shifts} of locally compact groups satisfy the \emph{Hopf dichotomy}: they are either totally dissipative or ergodic.

\section{Fundamentals}
\label{Section: Fundamentals}

Throughout this work, \(G\) stands for a locally compact second countable (lcsc) group, that is, a Polish group whose topology is locally compact. We will fix once and for all a left Haar measure \(\lambda\) on \(G\). A {\bf compact filtration} of \(G\) is a sequence \(K_{1}\subseteq K_{2}\subseteq\dotsm\) of compact subsets of \(G\) such that \(G=K_{1}\cup K_{2}\cup\dotsm\). Such a compact filtration is {\bf equicompact} if it has the property that every compact set \(C\subset G\) is contained in \(K_{n}\) for all sufficiently large \(n\in\mathbb{N}\). By a theorem of Struble \cite{struble1974}, every lcsc group \(G\) admits a {\bf compatible proper metric}, that is, a metric on \(G\) whose topology is the given topology of \(G\) and with respect to which closed ball are compact. In particular, an increasing sequence of balls in such a metric whose radii diverge to \(+\infty\) forms an equicompact filtration of \(G\).

\subsubsection{Borel \(G\)-Spaces}

A {\bf Borel \(G\)-space} is a standard Borel space \(X\) (whose \(\sigma\)-algebra is fixed but remains implicit) together with a Borel map
\[G\times X\to X,\quad\left(g,x\right)\mapsto g.x,\]
such that \(e.x=x\) for every \(x\in X\), where \(e\in G\) is the identity element, and \(gh.x=g.\left(h.x\right)\) for every \(g,h\in G\) and \(x\in X\). A Borel \(G\)-space is {\bf free} if \(g.x\neq x\) for every \(g\in G\backslash\left\{e\right\}\) and \(x\in X\).

\subsubsection{Cross Sections}
\label{sct:cross}

The notion of cross section for a Borel \(G\)-space is classical and is known for decades, but more recently it went through some useful improvements. See the survey \cite[\S3.2]{kechris2024}. In the following presentation we mostly follow Slutsky's treatment \cite[\S2]{slutsky2017}.

Let \(X\) be a Borel \(G\)-space. A {\bf cross section} (also called {\bf complete section} or {\bf countable section}) for \(X\) is a Borel set \(\mathfrak{C}\subseteq X\) whose intersection with every orbit is nonempty and countable. A cross section \(\mathfrak{C}\) is {\bf \(U\)-lacunary}, for some symmetric identity neighborhood \(U\subset G\), if the action map \(U\times\mathfrak{C}\to X\) is injective, i.e. \(U.w\cap U.z=\emptyset\) for all distinct \(w,z\in\mathfrak{C}\). By a theorem of Kechris \cite[Corollary 1.2]{kechris1992}, every Borel \(G\)-space admits a lacunary cross section.

A cross section \(\mathfrak{C}\) of \(X\) is {\bf \(K\)-cocompact}, for some compact identity neighborhood \(K\subset G\), if \(K.\mathfrak{C}=X\). By a theorem of C. Conley and L. Dufloux, strengthening the aforementioned theorem of Kechris, every Borel \(G\)-space admits a lacunary cocompact cross section. See \cite[Theorem 3.11]{kechris2024}. The following version is due to Slutsky \cite[Theorem 2.4]{slutsky2017}.

\begin{thm}
\label{thm:cross}
Let \(X\) be a Borel \(G\)-space. For every compact symmetric identity neighborhood \(U\subset G\), there exists a \(U\)-lacunary \(U^{2}\)-cocompact cross section for \(X\).
\end{thm}

\subsubsection{Voronoi Tessellations}
\label{sct:voronoi}

From a cross section of a free Borel \(G\)-space one may define in a standard way a {\bf Voronoi tessellation} of the space. We present this here following Slutsky \cite[\S4.1]{slutsky2017}.

Let \(X\) be a free Borel \(G\)-space and \(\mathfrak{C}\) a lacunary cross section for \(E_{G}^{X}\). Fix a compatible proper metric \(d\) on \(G\) and, as the action is free, let the map
\[d_{o}:E_{G}^{X}\to\mathbb{R}_{\geq 0},\quad d_{o}\left(x,g.x\right)=d\left(e,g\right).\]
For \(x\in X\) and \(r\geq 0\) denote the set
\[\mathfrak{C}_{r}\left(x\right):=\left\{ w\in\mathfrak{C}\cap G.x:d_{o}\left(x,w\right)\leq r\right\}.\]
The lacunarity of \(\mathfrak{C}\) shows that \(\mathfrak{C}_{r}\left(x\right)\) is a finite set and, since \(\mathfrak{C}\) is a cross section, for every \(x\in X\) there is some \(r\geq 0\) for which \(\mathfrak{C}_{r}\left(x\right)\) is nonempty. Consequently, the Borel function
\[r_{o}:X\to\mathbb{R}_{\geq0},\quad r_{o}\left(x\right)=\min\left\{ d_{o}\left(x,w\right):\left(x,w\right)\in\left(X\times\mathfrak{C}\right)\cap E_{G}^{X}\right\},\]
is well-defined. For \(x\in X\), consider the finite set
\[\mathfrak{C}\left(x\right):=\mathfrak{C}_{r_{o}\left(x\right)}\left(x\right)=\left\{w\in\mathfrak{C}\cap G.x:d_{o}\left(x,w\right)=r_{o}\left(x\right)\right\}.\]

The Voronoi tessellation \(\left\{T_{w}:w\in\mathfrak{C}\right\}\) will be designed to form a partition of \(X\), where we aim to allocate a point \(x\in X\) to a tile \(T_{w}\) if \(w\in\mathfrak{C}\left(x\right)\). While for a given \(x\) there can be multiple elements in \(\mathfrak{C}\left(x\right)\), there are at most finitely many such elements, so fix a Borel linear ordering \(\prec\) on \(X\) (say via a Borel isomorphism of \(X\) with \(\mathbb{R}\)) and define the allocation Borel map
\[\tau_{o}:X\to\mathfrak{C}\text{ by letting }\tau_{o}\left(x\right)\text{ be the }\prec\text{-least element of }\mathfrak{C}\left(x\right).\]
Accordingly, we define the Voronoi tessellation \(\left\{T_{w}:w\in\mathfrak{C}\right\}\) by
\[T_{w}:=\left\{x\in X:\tau_{o}\left(x\right)=w\right\},\quad w\in\mathfrak{C}.\]

\subsubsection{OERs, Smoothness, Idealism}

Suppose \(X\) is a Borel \(G\)-space. Then {\bf \emph{the} orbit equivalence relation} associated with \(X\) is
\[E_{G}^{X}:=\left\{\left(x,g.x\right):x\in X,g\in G\right\}\subseteq X\times X.\]
Since \(G\) is lcsc it is known that \(E_{G}^{X}\) is Borel, i.e. a Borel subset of \(X\times X\), and every class in \(E_{G}^{X}\), namely every \(G\)-orbit, is Borel (see e.g. \cite[Exercise 3.4.6, Theorem 3.3.2]{gao2008}).

The equivalence relation \(E_{G}^{X}\) is referred to as \emph{the} orbit equivalence relation associated with \(X\). It will be convenient for us to relax the notion of orbit equivalence relations as follows:

\begin{defn}
\label{dfn:OER}
Let \(X\) be a Borel \(G\)-space. An {\bf orbit equivalence relation} ({\bf OER}) on \(X\) is a \emph{Borel} subequivalence relation of \(E_{G}^{X}\).\\
An OER \(E\subseteq E_{G}^{X}\) is said to be {\bf \(U\)-positive}, for some identity neighborhood \(U\subset G\), if for every \(x\in X\) there exists \(g\in G\) with \(Ug.x\subseteq\left[x\right]_{E}\).
\end{defn}

The following notion of smoothness is central in the theory:

\begin{defn}[Smoothness]
An equivalence relation \(E\) on a standard Borel space \(X\) is {\bf smooth} if there is a Borel function \(s:X\rightarrow Y\), for some standard Borel space \(Y\), such that for every \(x,y\in X\),
\[\left(x,y\right)\in E\iff s\left(x\right)=s\left(y\right).\] 
\end{defn}

A smooth equivalence relation is always Borel, as it is the inverse image of the diagonal of \(Y\) under the Borel function \(\left(x,y\right)\mapsto\left(s\left(x\right),s\left(y\right)\right)\).

Given a Borel equivalence relation \(E\) on a standard Borel space \(X\), consider the space \(X/E\) of \(E\)-classes of points in \(X\). We recall that a {\bf \(\sigma\)-ideal} on a given set is a nonempty collection of subsets which is closed under taking subsets and countable unions.

\begin{defn}[Idealism]
A Borel equivalence relation \(E\) on a standard Borel space \(X\) is {\bf idealistic} if there is a Borel assignment
\[I:C\mapsto I_{C},\quad C\in X/E,\]
assigning to each \(E\)-class \(C\) a \(\sigma\)-ideal \(I_{C}\) on \(C\) such that \(C\notin I_{C}\). Here, \(I\) is being Borel in the sense that for every Borel set \(A\subseteq X\times X\), the set
\[A_{I}:=\big\{x\in X:\left(A\cap E\right)_{x}\in I_{\left[x\right]_{E}}\big\},\]
is Borel in \(X\), where we denote for a general set \(S\subseteq X\times X\),
\[S_{x}=\left\{y\in X:\left(x,y\right)\in S\right\}.\]
\end{defn}

The next proposition is analogous to \cite[Proposition 5.4.10]{gao2008}:

\begin{prop}
Every positive OER is idealistic.
\end{prop}

\begin{proof}
Let \(E\subseteq E_{G}^{X}\) be a positive OER. For \(\left[x\right]_{E}\in X/E\), define \(I_{\left[x\right]_{E}}\) by
\[S\in I_{\left[x\right]_{E}}\iff \lambda\left(g\in G:g.x\in S\right)=0,\text{ for whatever }S\subseteq\left[x\right]_{E}.\]
By the positivity of \(E\) it is clear that \(\left[x\right]_{E}\notin I_{\left[x\right]_{E}}\) and, from monotonicity and \(\sigma\)-additivity of \(\lambda\), we also see that \(I_{\left[x\right]_{E}}\) is a \(\sigma\)-ideal. In order to see that \(I\) is Borel, we note that for a Borel set \(A\subseteq X\times X\) we have by definition
\[A_{I}=\left\{ x\in X:\lambda\left(g\in G:g.x\in\left(A\cap E\right)_{x}\right)=0\right\}.\]
Since \(\left(x,g\right)\mapsto1_{A\cap E}\left(x,g.x\right)\) is a Borel function, from \cite[Theorem (17.25)]{kechris2012} it follows that also \(x\mapsto\lambda\left(g\in G:g.x\in\left(A\cap E\right)_{x}\right)\) is a Borel function, hence \(A_{I}\subseteq X\) is a Borel set.
\end{proof}

The following characterizations are due to Kechris (see \cite[Theorem 5.4.11]{gao2008}). For a comprehensive treatment we refer to \cite[\S5.4]{gao2008}.

\begin{thm}[Kechris]
Let \(E\) be an idealistic Borel equivalence relation (e.g. a positive OER) on a standard Borel space \(X\). TFAE:
\begin{enumerate}
    \item \(E\) is smooth.
    \item \(E\) admits a \emph{Borel selector}: a Borel function \(s:X\to X\) such that
    \[\left(x,s\left(x\right)\right)\in E\text{ and }\left(x,y\right)\in E\iff s\left(x\right)=s\left(y\right)\text{ for all }x,y\in X.\]
    \item \(E\) admits a \emph{Borel transversal}: a Borel subset \(T\subseteq X\) that intersects every class in exactly one point.
\end{enumerate}
When \(E\) is smooth with a Borel transversal \(T\), there can be found a Borel selector for \(E\) of the form \(s:X\to T\).
\end{thm}

A notable property of positive OERs is:

\begin{prop}
Let \(X\) be a free Borel \(G\)-space. If \(E\subseteq E_{G}^{X}\) is a positive OER, then every \(G\)-orbit contains at most countably many \(E\)-classes.
\end{prop}

\begin{proof}
Let \(G.x\) be a \(G\)-orbit and put \(G.x/E\) for the set of \(E\)-classes within \(G.x\). For every \(C\in G.x/E\) put \(G_{E}\left(C\right):=\left\{ g\in G:g.x\in C\right\}\). On one hand, those sets form a partition of \(G_{E}\left(x\right)\). On the other hand, for every \(C\in G.x/E\), if we fix any \(g_{C}\in G_{E}\left(C\right)\) then note that \(G_{E}\left(C\right)=G_{E}\left(g_{C}.x\right)g_{C}\). Since \(E\) is \(U\)-positive, \(G_{E}\left(C\right)\) contains a translation of \(U\). Since \(G\) is second countable there are at most countably many disjoint translations of \(U\), hence \(G.x/E\) is at most countable.
\end{proof}

\section{Hypercompact OERs}
\label{sct:hypercomp}

\subsection{Sectional-Hyperfiniteness}

Recall that a countable equivalence relation \(E\) on a standard Borel space \(X\) is {\bf hyperfinite} if it admits a {\bf finite filtration}, namely a sequence \(E_{1}\subseteq E_{2}\subseteq\dotsm\) of equivalence relations with \(E=E_{1}\cup E_{2}\cup\dotsm\) such that each \(E_{n}\) is finite, i.e. every \(E_{n}\)-class is finite. When \(E\) is uncountable, a definition for hyperfiniteness was presented by Connes, Feldman \& Weiss \cite[p. 447]{connes1981}:

\begin{defn}
\label{dfn:secthyperf}
Let \(X\) be a Borel \(G\)-space. We say that \(E_{G}^{X}\) is {\bf sectional-hyperfinite} if for one (hence every) cross section \(\mathfrak{C}\) of \(X\), the countable Borel equivalence relation \(E_{G}^{X}\cap\left(\mathfrak{C}\times\mathfrak{C}\right)\) is hyperfinite.
\end{defn}

It is worth noting, although we do not use it here, that by a well-known theorem of Dougherty, Jackson \& Kechris \cite[Theorem 5.1]{dougherty1994}, sectional-hyperfiniteness is equivalent to \emph{hypersmoothness} (recall that reducible-to-hyperfinite is hypersmooth \cite[\S8]{dougherty1994}).

The following proposition shows that sectional-hyperfiniteness is indeed independent on the choice of the cross section. It was communicated to us, together with its proof, by Alexander Kechris:

\begin{prop}
\label{prop:kechris}
Let \(E\) be a Borel equivalence relation on a standard Borel space \(X\), and let \(\mathfrak{C}\) and \(\mathfrak{C}'\) be Borel cross sections for \(E\). Then \(E\cap\left(\mathfrak{C}\times\mathfrak{C}\right)\) is hyperfinite if and only if \(E\cap\left(\mathfrak{C}'\times\mathfrak{C}'\right)\) is hyperfinite.
\end{prop}

\begin{rem}
\label{rem:meassecthyperf}
Connes, Feldman \& Weiss have proved Proposition \ref{prop:kechris} for nonsingular \(G\)-spaces, i.e. that being sectional-hyperfinite is independent on the choice of the cross section up to a null set \cite[Lemma 15, Corollary 16]{connes1981}.
\end{rem}

In the following proof we assume familiarity with the notion of \emph{Borel reduciblity} between Borel equivalence relations (see the survey \cite{kechris2024}).

\begin{proof}
It is known that a countable Borel equivalence relation reducible-to-hyperfinite is itself hyperfinite (see \cite[Proposition 5.2(2)]{dougherty1994}, \cite[Proposition 1.3(ii)]{jackson2002}). Then it is enough to show that \(E\cap\left(\mathfrak{C}\times\mathfrak{C}\right)\) is Borel reducible to \(E\cap\left(\mathfrak{C}'\times\mathfrak{C}'\right)\) whenever \(\mathfrak{C}\) and \(\mathfrak{C}'\) are cross sections for \(E\). In this setting, consider the Borel set
\[K:=E\cap\left(\mathfrak{C}\times\mathfrak{C}'\right).\]
For every \(x\in\mathfrak{C}\), the section \(\left\{x'\in\mathfrak{C}':\left(x,x'\right)\in K\right\}\) is countable and nonempty. Then by the Lusin--Novikov Uniformization Theorem (see \cite[Theorem (18.10)]{kechris2012}) there is a Borel function \(f:\mathfrak{C}\to\mathfrak{C}'\) such that \(\left(x,f\left(x\right)\right)\in K\) for all \(x\in\mathfrak{C}\). Thus, \(f\) is a Borel reduction of \(E\cap\left(\mathfrak{C}\times\mathfrak{C}\right)\) to \(E\cap\left(\mathfrak{C}'\times\mathfrak{C}'\right)\).
\end{proof}

\subsection{Hypercompactness}

We now came to define hypercompactness, which is a natural uncountable analog of hyperfiniteness. Suppose \(X\) is a Borel \(G\)-space and \(E\subseteq E_{G}^{X}\) is an OER. For every \(x\in X\) denote
\[G_{E}\left(x\right):=\left\{g\in G:\left(x,g.x\right)\in E\right\}.\]
Since \(E\) is Borel this is a Borel set in \(G\). One can routinely verify that
\begin{equation}
\label{eq:2}
G_{E}\left(g.x\right)=G_{E}\left(x\right)g^{-1}\text{ whenever }g\in G_{E}\left(x\right).
\end{equation}
Thus, an OER \(E\subseteq E_{G}^{X}\) is \(U\)-positive according to Definition \ref{dfn:OER} for some neighborhood \(U\subset G\), if for every \(x\in X\) there is \(g\in G\) with \(Ug\subseteq G_{E}\left(x\right)\). Let us now define the basic notions we will discuss.

\begin{defn}
Let \(X\) be a free Borel \(G\)-space. An OER \(E\subseteq E_{G}^{X}\) is {\bf compact} if \(G_{E}\left(x\right)\) is relatively compact for every \(x\in X\). It is {\bf uniformly compact} if, moreover, there is a compact set \(K\subset G\) such that \(G_{E}\left(x\right)\subseteq K\) for every \(x\in X\). In this case we may say that \(E\) is {\bf \(K\)-compact}.
\end{defn}

\begin{defn}
Let \(X\) be a free Borel \(G\)-space. A {\bf compact filtration} of \(E_{G}^{X}\) is an increasing sequence of compact OERs \(E_{1}\subseteq E_{2}\subseteq\dotsm\) with \(E_{G}^{X}=E_{1}\cup E_{2}\cup\dotsm\). Such a compact filtration is {\bf equicompact} if for every \(x\in X\), the compact filtration \(G_{E_{1}}\left(x\right)\subseteq G_{E_{2}}\left(x\right)\subseteq\dotsm\) of \(G\) is equicompact, i.e. every compact set in \(G\) is eventually contained in \(G_{E_{n}}\left(x\right)\).\footnote{As defined in Section \ref{Section: Fundamentals}.}
\end{defn}

\begin{defn}
\label{def:hypercomp}
Let \(X\) be a free Borel \(G\)-space. \(E_{G}^{X}\) is {\bf hypercompact} if either of the following equivalent condition holds:
\begin{enumerate}
    \item \(E_{G}^{X}\) admits a compact filtration.
    \item \(E_{G}^{X}\) admits an equicompact filtration.
    \item \(E_{G}^{X}\) admits an equicompact filtration of uniformly compact OERs.
\end{enumerate}
\end{defn}

The equivalence of the three conditions defining hypercompactness will be proved in the forthcoming Theorem \ref{thm:CFW+}.

\begin{rem}
Unlike the notion of hyperfiniteness which is defined plainly for any countable equivalence relation, the notion of hypercompactness is restricted to orbit equivalence relations and relies on the canonical topology of the group which acts freely. We do not know whether one can define meaningfully hypercompactness for general equivalence relations.
\end{rem}

\begin{exm}[Free transitive actions are hypercompact]
It is an elementary fact that the action of every countable group \(G\) on itself is hyperfinite. Indeed, we may define for every \(n\in\mathbb{N}\) a partition \(T_{n}\) of \(G\) into finite subsets of \(G\), each of which of size \(2^{n}\). We can do this in such a way that \(T_{n+1}\) is coarser than \(T_{n}\) (i.e. every element of \(T_{n}\) is contained in an element of \(T_{n+1}\)) for every \(n\in\mathbb{N}\). Now letting \(E_{n}\) be the equivalence relation whose classes are the elements of \(T_{n}\) for every \(n\in\mathbb{N}\), it is easy to see that \(E_{0}\subseteq E_{1}\subseteq\dotsm\) forms a finite filtration of \(E_{G}^{G}=G\times G\).

Let us show that \(E_{G}^{G}=G\times G\) is hypercompact for every lcsc group \(G\) using essentially the same argument. Fix some compatible proper metric \(d\) on \(G\). Start by picking a countable set \(A:=\left\{a_{1},a_{2},\dotsc\right\}\subset G\) which is \(1\)-discrete (i.e. \(d\left(a_{i},a_{j}\right)\geq 1\) for all distinct \(a_{i},a_{j}\in A\)) and \(2\)-dense (i.e. \(\mathrm{dist}\left(g,A\right)<2\) for every \(g\in G\)). Form a Voronoi tessellation \(\left\{T_{1},T_{2},\dotsc\right\}\) so that \(T_{i}\) consists of points which are \(1\)-distant from \(a_{i}\), for \(i=1,2,\dotsc\) (just as described in Section \ref{sct:voronoi}), from which we obtain the equivalence relation \(E_{0}\) whose classes are \(\left\{T_{1},T_{2},\dotsc\right\}\). Now for every \(n\in\mathbb{N}\) let \(E_{n}\) be the equivalence relation whose classes are
\[\left\{ T_{1}\cup\dotsm T_{2^{n}},T_{2^{n}+1}\cup\dotsm\cup T_{2^{n+1}},T_{2^{n+2}}\cup\dotsm\cup T_{2^{n+2}+1},\dotsc\right\}.\]
Since \(A\) is uniformly discrete, every ball in \(G\) is covered by finitely many of the \(T_{i}\)'s, hence \(E_{0}\subseteq E_{1}\subseteq\dotsm\) is an equicompact filtration of \(E_{G}^{G}=G\times G\).
\end{exm}

\begin{prop}
Let \(X\) be a free Borel \(G\)-space. If \(E_{G}^{X}\) is hypercompact then it admits an equicompact filtration consisting of positive OERs.
\end{prop}

\begin{proof}
Fix an equicompact filtration \(F_{1}\subseteq F_{2}\subseteq\dotsm\) of \(E_{G}^{X}\). For \(x\in X\) put
\[\varphi_{1}\left(x\right)=\inf\left\{ m\geq 1:\lambda\left(G_{F_{m}}\left(x\right)\right)>0\right\}.\]
Since \(E_{G}^{X}\) is clearly positive, necessarily \(\varphi_{1}\left(x\right)<+\infty\) for every \(x\in X\). For every \(m\), since \(F_{m}\) is Borel, from \cite[Theorem (17.24)]{kechris2012} it follows that \(x\mapsto\lambda\left(G_{F_{m}}\left(x\right)\right)\) is a Borel function, hence so is \(\varphi_{1}\). Define then \(E_{1}\subseteq E\) by
\[\left(x,y\right)\in E_{1}\iff\left(x,y\right)\in F_{\varphi_{1}\left(x\right)}.\]
It is clear that \(E_{1}\subseteq X\times X\) is a Borel set. Note that if \(\left(x,y\right)\in F_{m}\) for some \(m\in\mathbb{N}\) and \(\lambda\left(G_{F_{m}}\left(x\right)\right)>0\), then in light of \eqref{eq:2} and the quasi-invariance of \(\lambda\) to multiplication from the right also \(\lambda\left(G_{F_{m}}\left(y\right)\right)>0\), so it easily follows that \(\varphi_{1}\left(x\right)=\varphi_{1}\left(y\right)\). This implies that \(E_{1}\) is an equivalence relation. We also note that \(G_{E_{1}}\left(x\right)=G_{F_{\varphi_{1}\left(x\right)}}\left(x\right)\), hence \(E_{1}\) is compact.

Proceeding by induction, for \(n\geq 2\) put
\[\varphi_{n}\left(x\right)=\inf\left\{ m\geq\varphi_{n-1}\left(x\right)+1:\lambda\left(G_{F_{m}}\left(x\right)\right)>0\right\},\]
and define \(E_{n}\) by
\[\left(x,y\right)\in E_{n}\iff\left(x,y\right)\in F_{\varphi_{n}\left(x\right)}.\]
It is then routine to verify that \(E_{1}\subseteq E_{2}\subseteq\dotsm\) forms an equicompact filtration of \(E_{G}^{X}\) consisting of positive OERs.
\end{proof}

\subsection{A Characterization of Hypercompactness}

We now formulate our first main result, of which Theorem \ref{thm:CFW} is a particular case. It also justifies the three equivalent forms of hypercompactness as in Definition \ref{def:hypercomp}.

\begin{thm}
\label{thm:CFW+}
For a free Borel \(G\)-space \(X\) TFAE:
\begin{enumerate}
    \item \(E_{G}^{X}\) admits a compact filtration.
    \item \(E_{G}^{X}\) is sectional-hyperfinite.
    \item \(E_{G}^{X}\) admits an equicompact filtration.
    \item \(E_{G}^{X}\) admits an equicompact filtration of uniformly compact OERs.
\end{enumerate}
Accordingly, \(E_{G}^{X}\) is {\bf hypercompact} when these equivalent conditions hold.
\end{thm}

The hard part of Theorem \ref{thm:CFW+} is the implication (2)\(\implies\)(3), and its proof relies on the existence of compact OER with sufficient regularity:

\begin{defn}
\label{dfn:eu}
Let \(X\) be a free Borel \(G\)-space, \(U\subset G\) a compact symmetric identity neighborhood, and \(\mathfrak{C}\) a \(U\)-lacunary \(U^{2}\)-cocompact cross section (recall Theorem \ref{thm:cross}) with Voronoi tessellation \(\left\{T_{w}:w\in\mathfrak{C}\right\}\). Define
\[E_{U}:=\{\left(x,y\right)\in X\times X:\exists w\in\mathfrak{C},\,\,x,y\in T_{w}\};\]
thus, \(E_{U}\) is defined by having the equivalence classes \(\left\{T_{w}:w\in\mathfrak{C}\right\}\).
\end{defn}

\begin{prop}
Every \(E_{U}\) as in Definition \ref{dfn:eu} is a \(U\)-positive and \(U^{4}\)-compact smooth OER of \(E_{G}^{X}\).
\end{prop}

\begin{proof}
First, \(E_{U}\) is an OER of \(E_{G}^{X}\) by the definition of the Voronoi tesselation. By construction, the allocation map \(\tau_{o}:X\to\mathfrak{C}\) is a Borel selector into \(\mathfrak{C}\), and thus \(E_{U}\) is smooth. By the \(U\)-lacunarity of \(\mathfrak{C}\) we see that \(E_{U}\) is \(U\)-positive. By the \(U^{2}\)-cocompactness of \(\mathfrak{C}\) we see that \(G_{E_{U}}\left(x\right)\subseteq U^{4}\) for every \(x\in X\), thus \(E_{U}\) is \(U^{4}\)-compact.
\end{proof}

The following technical lemma will be important to proving Theorem \ref{thm:CFW+}.

\begin{lem}
\label{lem:eu}
Every compact filtration \(E_{1}\subseteq E_{2} \subseteq \dotsm\) of \(E_{G}^{X}\) such that \(E_U \subseteq E_{1}\) for some \(E_{U}\) as in Definition~\ref{dfn:eu}, is equicompact.
\end{lem}

\begin{proof}
Let \(E_{U}\subseteq E_{1}\subseteq E_{2}\subseteq\dotsm\) be a compact filtration. Abbreviate \(G_{n}\left(\cdot\right)=G_{E_{n}}\left(\cdot\right)\) and \(\left[\cdot\right]_{n}=\left[\cdot\right]_{E_{n}}\). The key property we will use is:
\begin{equation}
\label{eq:keyprop}
\begin{aligned}
\text{For every } x \in X \text{ and every sufficiently large } n\in \mathbb{N},\\G_{n}\left(x\right)
\text{contains an identity neighborhood in } G.
\end{aligned}
\end{equation}
Let us prove \eqref{eq:keyprop}. Fix an arbitrary \(x\in X\). Since \(\mathfrak{C}_{r_{o}\left(x\right)+1}\left(x\right)\) is a finite set containing \(\mathfrak{C}\left(x\right)\), there exists \(\epsilon=\epsilon\left(x\right)>0\) sufficiently small with
\[\mathfrak{C}_{r_{o}\left(x\right)+\epsilon}\left(x\right)=\mathfrak{C}_{r_{o}\left(x\right)}\left(x\right)=\mathfrak{C}\left(x\right).\]
Look at \(B_{\epsilon/2}\left(e\right)\), the open ball of radius \(\epsilon/2\) around the identity \(e\in G\) with respect to the compatible proper metric defining the Voronoi tessellation, and we aim to show that \(B_{\epsilon/2}\left(e\right)\subseteq G_{n}\left(x\right)\) for every sufficiently large \(n\in\mathbb{N}\). Let \(g\in B_{\epsilon/2}\left(e\right)\) be arbitrary. First note that for whatever \(w\in\mathfrak{C}\left(x\right)\) we have
\[d_{o}\left(g.x,w\right)\leq d_{o}\left(x,w\right)+d_{o}\left(x,g.x\right)<r_{o}\left(x\right)+\epsilon/2,\]
hence by the definition of \(r_{o}\) we have
\[r_{o}\left(g.x\right)<r_{o}\left(x\right)+\epsilon/2.\]
Now pick any \(w\in\mathfrak{C}\left(g.x\right)\), namely \(r_{o}\left(g.x\right)=d_{o}\left(g.x,w\right)\), and we see that
\begin{align*}
d_{o}\left(x,w\right)
&\leq d_{o}\left(x,g.x\right)+d_{o}\left(g.x,w\right)\\
&<\epsilon/2+r_{o}\left(g.x\right)<r_{o}\left(x\right)+\epsilon.
\end{align*}
This means that \(w\in\mathfrak{C}_{r_{o}\left(x\right)+\epsilon}\left(x\right)=\mathfrak{C}\left(x\right)\), so we deduce that \(\mathfrak{C}\left(g.x\right)\subseteq\mathfrak{C}\left(x\right)\). Now since \(\mathfrak{C}\left(x\right)\) is finite and \(E_{1}\subseteq E_{2}\subseteq\dotsm\) is a filtration of \(E_{G}^{X}\), there can be found \(n=n\left(x\right)\in\mathbb{N}\) sufficiently large such that \(\left[\mathfrak{C}\left(x\right)\right]_{n}\subseteq\left[x\right]_{n}\), hence
\[g.x\in\left[\mathfrak{C}\left(g.x\right)\right]_{n}\subseteq\left[\mathfrak{C}\left(x\right)\right]_{n}\subseteq\left[x\right]_{n}.\]
We found that \(B_{\epsilon/2}\left(e\right)\subseteq G_{n}\left(x\right)\), establishing \eqref{eq:keyprop}.

We now deduce the equicompactness of \(E_{U}\subseteq E_{1}\subseteq E_{2}\subseteq\dotsm\). Let \(x\in X\) be arbitrary and let \(C\subset G\) be an arbitrary compact set. For every \(c\in C\), from \eqref{eq:keyprop} we may pick \(n=n\left(x,c\right)\in\mathbb{N}\) such that \(G_{n}\left(c.x\right)\) contains some identity neighborhood \(V\). Enlarging \(n\) if necessary, we may assume that \(\left(x,c.x\right)\in E_{n}\), that is \(c\in G_{n}\left(x\right)\), and using \eqref{eq:2} it follows that
\[Vc\subseteq G_{E_{U}}\left(c.x\right)c=G_{E_{U}}\left(x\right),\text{ hence }c\in\mathrm{Int}\left(G_{n}\left(x\right)\right).\]
We deduce that \(C\subseteq\bigcup_{n\in\mathbb{N}}\mathrm{Int}\left(G_{n}\left(x\right)\right)\), and since \(C\) is compact there exists \(n_{o}\in\mathbb{N}\) such that \(C\subseteq G_{n}\left(x\right)\) for all \(n>n_{o}\).
\end{proof}

\begin{proof}[Proof of Theorem \ref{thm:CFW+}]
Let \(X\) be a free Borel \(G\)-space with OER \(E_{G}^{X}\).

(1)\(\implies\)(2): Fix a compact filtration \(E_{1}\subseteq E_{2}\subseteq\dotsm\) of \(E_{G}^{X}\), and let \(\mathfrak{C}\) be a \(U\)-lacunary cross section for \(E_{G}^{X}\) (we do not use cocompactness here). Put \(F_{n}:=E_{n}\cap\left(\mathfrak{C}\times\mathfrak{C}\right)\) for \(n\in\mathbb{N}\). Clearly \(F_{1}\subseteq F_{2}\subseteq\dotsm\) is a filtration of \(E_{G}^{X}\cap\left(\mathfrak{C}\times\mathfrak{C}\right)\), so we have left to show that \(F_{n}\) is a finite equivalence relation for every \(n\in\mathbb{N}\). Indeed, since the action is free, for every \(x\in X\), the cardinality of the \(F_{n}\)-class of \(x\) is the cardinality of the set \(G_{F_{n}}\left(x\right)\). However, \(G_{F_{n}}\left(x\right)\) is \(U\)-discrete and contained in \(G_{E_{n}}\left(x\right)\), which is relatively compact, and therefore \(G_{F_{n}}\left(x\right)\) is finite.

(2)\(\implies\)(3): Let \(U\subset G\) be some identity neighborhood and, using Theorem \ref{thm:cross}, pick a \(U\)-lacunary \(U^{2}\)-cocompact cross section \(\mathfrak{C}\) for \(X\), and let \(E_{U}\subseteq E_{G}^{X}\) be as in Definition \ref{dfn:eu}. By the assumption that \(E_{G}^{X}\) is sectional-hyperfinite, \(E_{G}^{X}\cap\left(\mathfrak{C}\times\mathfrak{C}\right)\) is hyperfinite. Let then \(F_{1}\subseteq F_{2}\subseteq\dotsm\) be a finite filtration of \(E_{G}^{X}\cap\left(\mathfrak{C}\times\mathfrak{C}\right)\). With the Borel selector \(\tau_{o}:X\to\mathfrak{C}\) of \(E_{U}\), define
\[E_{n}:=\left\{ \left(x,y\right)\in X\times X:\left(\tau_{o}\left(x\right),\tau_{o}\left(y\right)\right)\in F_{n}\right\},\quad n\in\mathbb{N}.\]
It can be directly verified that \(E_{1}\subseteq E_{2}\subseteq\dotsm\) forms a filtration of \(E_{G}^{X}\). Since \(\left(x,y\right)\in E_{U}\iff\tau_{o}\left(x\right)=\tau_{o}\left(y\right)\) it is clear that \(E_{U}\subseteq E_{1}\). Let us verify that \(E_{n}\) is compact for each \(n\in\mathbb{N}\). To this end we first observe that for every \(x\in X\), the \(F_{n}\)-equivalence class of \(\tau_{o}\left(x\right)\) is of the form
\[\left[\tau_{o}\left(x\right)\right]_{F_{n}}=\left\{ \tau_{o}\left(s.x\right):s\in S_{n}\left(x\right)\right\}\text{ for some finite set }S_{n}\left(x\right)\subset G.\]
Indeed, fix any \(s\in G\) with \(\tau_{o}\left(x\right)=s.x\), and then every \(y\in \left[\tau_{o}\left(x\right)\right]_{F_{n}}\) is of the general form \(y=g.\tau_{o}\left(x\right)=gs.x\) for a particular \(g\in G\), and since \(y\in\mathfrak{C}\) we have \(y=gs.x=\tau_{o}\left(gs.x\right)\). We can now deduce the compactness of \(E_{n}\), since for every \(x\in X\) we have
\begin{align*}
G_{E_{n}}\left(x\right)
&=\left\{ g\in G:\tau_{o}\left(g.x\right)\in \left[\tau_{o}\left(x\right)\right]_{F_{n}}\right\}\\
&=\bigcup\nolimits_{s\in S_{n}\left(x\right)}\left\{ g\in G:\tau_{o}\left(g.x\right)=\tau_{o}\left(s.x\right)\right\}\\
&=\bigcup\nolimits_{s\in S_{n}\left(x\right)}\left\{ g\in G:\left(g.x,s.x\right)\in E_{U}\right\}\\
&=\bigcup\nolimits_{s\in S_{n}\left(x\right)}G_{E_{U}}\left(s.x\right)=G_{E_{U}}\left(x\right)S_{n}\left(x\right)^{-1}\subseteq U^{4}S_{n}\left(x\right)^{-1},
\end{align*}
and since \(U\) is compact while \(S_{n}\left(x\right)\) is finite, it follows that \(G_{E_{n}}\left(x\right)\) is relatively compact. Finally, since \(E_{G}^{X}\) admits a compact filtration whose first element is \(E_{U}\), by Lemma \ref{lem:eu} it is an equicompact filtration.

(3)\(\implies\)(4): Let \(E_{1}\subseteq E_{2}\subseteq\dotsm\) be an equicompact filtration of \(E_{G}^{X}\). Fix an equicompact filtration \(K_{1}\subseteq K_{2}\subseteq\dotsm\) of \(G\). For every \(n\in\mathbb{N}\), define
\[E_{n}^{m}:=\left\{ \left(x,y\right)\in E_{n}:G_{E_{n}}\left(x\right)\cup G_{E_{n}}\left(y\right)\subseteq K_{m}\right\},\quad m\in\mathbb{N}.\]
Then each \(E_{n}^{m}\subseteq E_{n}\) is a \(K_{m}\)-compact OER, and we have \(E_{n}=\bigcup_{m\in\mathbb{N}}E_{n}^{m}\). Note also that \(E_{m}^{m}\subseteq E_{m}^{m+1}\subseteq E_{m+1}^{m+1}\) for every \(m\in\mathbb{N}\). It then follows that \(E_{1}^{1}\subseteq E_{2}^{2}\subseteq\dotsm\) is an equicompact filtration such that each \(E_{m}^{m}\) is \(K_{m}\)-compact, and thus uniformly compact.

(4)\(\implies\)(1): This is obvious.
\end{proof}

\section{Measure-Hypercompact OERs}

Let us recall the basic setup of the nonsingular ergodic theory. A {\bf standard measure space} is a measure space \(\left(X,\mu\right)\) such that \(X\) is a standard Borel space and \(\mu\) is a Borel \(\sigma\)-finite measure on \(X\). A Borel set \(A\subseteq X\) is said to be {\bf \(\mu\)-null} if \(\mu\left(A\right)=0\) and it is said to be {\bf \(\mu\)-conull} if \(\mu\left(X\backslash A\right)=0\). A Borel property of the points of \(X\) will be said to hold {\bf modulo \(\mu\)} or {\bf \(\mu\)-a.e.} if the set of points satisfying this property is \(\mu\)-conull.

A {\bf nonsingular \(G\)-space} is a standard measure space \(\left(X,\mu\right)\) such that \(X\) is a Borel \(G\)-space and \(\mu\) is quasi-invariant to the action of \(G\) on \(X\); that is, the measures \(\mu\circ g^{-1}\) and \(\mu\) are mutually absolutely continuous for every \(g\in G\). A particular case of a nonsingular \(G\)-space is {\bf measure preserving \(G\)-space}, in which we further have \(\mu\circ g^{-1}=\mu\) for every \(g\in G\). When \(\mu\) is a probability measure, we stress this by using the terminology {\bf nonsingular probability \(G\)-space} or {\bf probability preserving \(G\)-space}.

\begin{defn}
\label{dfn:meashyper}
Let \(\left(X,\mu\right)\) be a nonsingular \(G\)-space. We say that \(E_{G}^{X}\) is {\bf \(\mu\)-sectional-hyperfinite} or {\bf \(\mu\)-hypercompact} if there exists a \(\mu\)-conull \(G\)-invariant set \(X_{o}\subseteq X\) such that
\[E_{G}^{X_{o}}:=E_{G}^{X}\cap\left(X_{o}\times X_{o}\right)\]
is sectional hyperfinite or hypercompact, respectively.
\end{defn}

When the measure \(\mu\) is clear in the context, we may call these notions by {\bf measure-sectional-hyperfinite} or {\bf measure-hypercompact}. As mentioned in Remark \ref{rem:meassecthyperf}, Connes, Feldman \& Weiss proved that being measure-sectional-hyperfinite with respect to one cross section implies the same for all other cross sections.

The requirement in Definition \ref{dfn:meashyper} that \(X_{o}\) would be \(G\)-invariant is necessary to defining hypercompactness, as this notion is define exclusively for OERs, but it is unnecessary to define sectional-hyperfiniteness, since hyperfiniteness is defined for general countable equivalence relations. Note however that also in measure-sectional-hyperfiniteness we may assume that it is the case that \(X_{o}\) is \(G\)-invariant; indeed, when this is not the case, there is a Borel set \(X_{oo}\subseteq X_{o}\) such that the \(G\)-invariant set \(G.X_{oo}\) is Borel and \(\mu\)-conull.\footnote{Proof: by Varadarajan's compact model theorem \cite[Ch.~V, \S3, Theorem~5.7]{varadarajan1968} it is sufficient to assume that \(G\) acts continuously. Passing to a finite equivalent measure, the tightness of finite measures (see \cite[\S17.C]{kechris2012}) ensures that \(X_{o}\) contains a \(\sigma\)-compact \(\mu\)-conull set \(X_{oo}\). Since \(G\) and \(X_{oo}\) are both \(\sigma\)-compact, \(G.X_{oo}\) is \(\sigma\)-compact hence Borel.} Now \(E_{G}^{G.X_{oo}}\) remains sectional-hyperfinite, because every cross section for \(X_{o}\) can be restricted to a cross section for \(X_{oo}\) and, generally speaking, every cross section for \(X_{oo}\) is a cross section for \(G.X_{oo}\).

In light of this discussion, Corollary \ref{cor:OW} is nothing but a reformulation of the celebrated Connes--Feldman--Weiss Theorem using Theorem \ref{thm:CFW}:

\begin{proof}[Proof of Corollary \ref{cor:OW}]
Let \(G\) be an amenable group and \(\left(X,\mu\right)\) be a nonsingular \(G\)-space. By the Connes--Feldman--Weiss Theorem there exists a \(\mu\)-conull set \(X_{o}\subseteq X\) such that \(E_{G}^{X_{o}}\) is sectional-hyperfinite. By the above discussion we may assume that \(X_{o}\) is \(G\)-invariant, thus \(E_{G}^{X_{o}}\) is a sectional hypercompact OER. Then by Theorem \ref{thm:CFW} we deduce that \(E_{G}^{X_{o}}\) is hypercompact, thus \(E_{G}^{X}\) is \(\mu\)-hypercompact.
\end{proof}

\section{Conditional Expectation on Compact OERs}

Here we introduce a formula for conditional expectation on the \(\sigma\)-algebra of sets that are invariant to a compact OER. It will be useful both in relating hypercompactness to amenability and in establishing the Random Ratio Ergodic Theorem.

Let \(X\) be a Borel \(G\)-space. For an OER \(E\subseteq E_{G}^{X}\) we denote by
\[\mathrm{Inv}\left(E\right)\]
the \(E\)-invariant \(\sigma\)-algebra whose elements are the Borel sets in \(X\) which are unions of \(E\)-classes. Thus, a Borel set \(A\subseteq X\) is in \(\mathrm{Inv}\left(E\right)\) if and only if for every \(x\in X\), either the entire \(E\)-class of \(x\) is in \(A\) or that the entire \(E\)-class of \(x\) is outside \(A\). The \(E_{G}^{X}\)-invariant \(\sigma\)-algebra will be abbreviated by
\[\mathrm{Inv}_{G}\left(X\right):=\mathrm{Inv}\left(E_{G}^{X}\right).\]
Thus, a Borel set \(A\subseteq X\) is in \(\mathrm{Inv}_{G}\left(X\right)\) if it is a \(G\)-invariant set.

From a given filtration \(E_{1}\subseteq E_{2}\subseteq\dotsm\) of \(E_{G}^{X}\), we obtain an {\bf approximation} of \(\mathrm{Inv}_{G}\left(X\right)\) by the sequence
\[\mathrm{Inv}\left(E_{1}\right)\supset\mathrm{Inv}\left(E_{2}\right)\supset\dotsm,\]
that satisfies
\[\mathrm{Inv}_{G}\left(X\right)=\mathrm{Inv}\left(E_{1}\right)\cap\mathrm{Inv}\left(E_{2}\right)\cap\dotsm.\]

Suppose \(\left(X,\mu\right)\) is a nonsingular \(G\)-space. It then has an associated {\bf Radon--Nikodym cocycle}, which is a Borel function
\begin{equation}
\label{eq:rncoc}
\nabla:G\times X\to\mathbb{R}_{>0},\quad\nabla:\left(g,x\right)\mapsto\nabla_{g}\left(x\right),
\end{equation}
that satisfies the cocycle identity
\[\nabla_{gh}\left(x\right)=\nabla_{g}\left(h.x\right)\nabla_{h}\left(x\right),\quad g,h\in G,\,x\in X,\]
and has the property
\[\nabla_{g}\left(\cdot\right)=\frac{d\mu\circ g}{d\mu}\left(\cdot\right)\text{ in }L^{1}\left(X,\mu\right)\text{ for each }g\in G.\]
The fact that there can be found a version of the Radon--Nikodym cocycle that satisfies the cocycle identity pointwise is nontrivial, and is due to the Mackey Cocycle Theorem (see e.g. \cite[Lemma 5.26, p. 179]{varadarajan1968}).

Using nothing but the Fubini Theorem, the Radon--Nikodym property of \(\nabla\) can be put generally in the following formula:
\begin{equation}
\label{eq:nabla}\tag{\(\dagger\)}
\begin{aligned}
&\iint\nolimits_{G\times X}\nabla_{g}\left(x\right)f_{0}\left(g.x\right)f_{1}\left(x\right)\varphi\left(g\right)d\lambda\otimes\mu\left(g,x\right)\\
&\qquad\quad=\iint\nolimits_{G\times X}f_{0}\left(x\right)f_{1}\left(g^{-1}.x\right)\varphi\left(g\right)d\lambda\otimes\mu\left(g,x\right),
\end{aligned}
\end{equation}
\[\text{ for all Borel functions } f_{0},f_{1}:X\to\left[0,\infty\right)\text{ and }\varphi:G\to\left[0,\infty\right).\]

We introduce the main formula we need for conditional expectation on a compact equivalence relation. For finite OERs in the context of countable acting groups, this formula was mentioned in \cite[\S1]{danilenko2019}  (cf. the {\it generalized Bayes' law} in \cite[Chapter 2, \S7, pp. 230-232]{shiryaev1996}).

Let \(E\subseteq E_{G}^{X}\) be a compact OER. With the Radon--Nikodym cocycle \eqref{eq:rncoc}, define an operator of measurable functions on \(X\) by
\[S^{E}:f\mapsto S_{f}^{E},\quad S_{f}^{E}\left(x\right):=\int_{G_{E}\left(x\right)}\nabla_{g}\left(x\right)f\left(g.x\right)d\lambda\left(g\right).\]

\begin{prop}
\label{prop:condexp}
For a compact positive OER \(E\subseteq E_{G}^{X}\), the conditional expectation of every \(f\in L^{1}\left(X,\mu\right)\) with respect to \(\mathrm{Inv}\left(E\right)\) has the formula
\[\mathbb{E}\left(f\mid\mathrm{Inv}\left(E\right)\right)\left(x\right)=S_{f}^{E}\left(x\right)/S_{1}^{E}\left(x\right)=\frac{\int_{G_{E}\left(x\right)}\nabla_{g}\left(x\right)f\left(g.x\right)d\lambda\left(g\right)}{\int_{G_{E}\left(x\right)}\nabla_{g}\left(x\right)d\lambda\left(g\right)},\]
for \(\mu\)-a.e. \(x\in X\) (depending on \(f\)).
\end{prop}

We will start by a simple lemma that demonstrates that compact OERs are \emph{dissipative} in nature (cf. \cite[\S1.0.1]{avraham2024hopf}).

\begin{lem}
Let \(E\subseteq E_{G}^{X}\) be a compact positive OER. Then for every \(f\in L^{1}\left(X,\mu\right)\) we have \(S_{f}^{E}\left(x\right)<+\infty\) for \(\mu\)-a.e. \(x\in X\) (depending on \(f\)).
\end{lem}

\begin{proof}
Pick a compatible proper metric on \(G\) and for \(r>0\) let \(B_{r}\) be the ball of radius \(r\) around the identity with respect to this metric. Recalling the identity \eqref{eq:nabla}, for every \(f\in L^{1}\left(X,\mu\right)\) and every \(r>0\) we have
\begin{align*}
\iint\nolimits_{B_{r}\times X}\nabla_{g}\left(x\right)f\left(g.x\right)d\lambda\otimes\mu\left(g,x\right)
& =\iint\nolimits_{G\times X}f\left(x\right)1_{B_{r}}\left(g^{-1}\right)d\lambda\otimes\mu\left(g,x\right)\\
& =\lambda\big(B_{r}^{-1}\big)\left\Vert f\right\Vert_{L^{1}\left(X,\mu\right)}<+\infty.
\end{align*}
Since this holds for every \(r>0\), there exists a \(\mu\)-conull set \(A_{f}\) such that
\[\int_{B_{r}}\nabla_{g}\left(x\right)f\left(g.x\right)d\lambda\left(g\right)<+\infty\text{ for every }r>0\text{ and every }x\in A_{f}.\]
Then for every \(x\in A_{f}\), since \(G_{E}\left(x\right)\) is relatively compact it is contained in some \(B_{r}\) for some sufficiently large \(r>0\), hence \(S_{f}^{E}\left(x\right)<+\infty\).
\end{proof}

\begin{proof}[Proof of Proposition \ref{prop:condexp}]
Note that for every \(x\in X\), if \(g\in G_{E}\left(x\right)\) then \(1_{E}\left(g.x,hg.x\right)=1_{E}\left(x,hg.x\right)\) for all \(h\in G\). Let \(\varDelta:G\to\mathbb{R}_{>0}\) be the modular function of \(G\) with respect to \(\lambda\). From the cocycle property of \(\nabla\) it follows that whenever \(g\in G_{E}\left(x\right)\),
\begin{equation}
\label{eq:S}\tag{I}
\small
\begin{aligned}
& S_{f}^{E}\left(g.x\right)\\
&\quad=\int_{G}1_{E}\left(g.x,hg.x\right)\nabla_{h}\left(g.x\right)f\left(hg.x\right)d\lambda\left(h\right)\\
&\quad=\nabla_{g}\left(x\right)^{-1}\int_{G}1_{E}\left(x,hg.x\right)\nabla_{hg}\left(x\right)f\left(hg.x\right)d\lambda\left(h\right)\\
&\quad=\nabla_{g}\left(x\right)^{-1}\varDelta\left(g\right)\int_{G}1_{E}\left(x,h.x\right)\nabla_{h}\left(x\right)f\left(h.x\right)d\lambda\left(h\right)\\
&=\nabla_{g}\left(x\right)^{-1}\varDelta\left(g\right)S_{f}^{E}\left(x\right).
\end{aligned}
\end{equation}
In particular, the function \(X\to\mathbb{R}_{>0}\), \(x\mapsto S_{f}^{E}\left(x\right)/S_{1}^{E}\left(x\right)\), is \(E\)-invariant, i.e. \(\mathrm{Inv}\left(E\right)\)-measurable. From all the above we obtain the following identity:
\begin{equation}
\label{eq:Sf}\tag{II}
\small
\begin{aligned}
&\int_{X}f\left(x\right)d\mu\left(x\right)\\
&\quad=\int_{X}f\left(x\right)S_{1}^{E}\left(x\right)S_{1}^{E}\left(x\right)^{-1}d\mu\left(x\right)\\
&\quad=\iint\nolimits_{G\times X}\nabla_{g}\left(x\right)f\left(x\right)1_{E}\left(x,g.x\right)S_{1}^{E}\left(x\right)^{-1}d\lambda\otimes\mu\left(g,x\right)\\
{\scriptstyle \eqref{eq:nabla}}&\quad=\iint\nolimits_{G\times X}f\left(g^{-1}.x\right)1_{E}\left(g^{-1}.x,x\right)S_{1}^{E}\left(g^{-1}.x\right)^{-1}d\lambda\otimes\mu\left(g,x\right)\\
{\scriptstyle \eqref{eq:S}}&\quad=\iint\nolimits_{G\times X}f\left(g^{-1}.x\right)1_{E}\left(g^{-1}.x,x\right)\nabla_{g^{-1}}\left(x\right)\varDelta\left(g\right)S_{1}^{E}\left(x\right)^{-1}d\lambda\otimes\mu\left(g,x\right)\\
&\quad=\iint\nolimits_{G\times X}f\left(g.x\right)1_{E}\left(g.x,x\right)\nabla_{g}\left(x\right)S_{1}^{E}\left(x\right)^{-1}d\lambda\otimes\mu\left(g,x\right)\\
&=\int_{X}S_{f}^{E}\left(x\right)S_{1}^{E}\left(x\right)^{-1}d\mu\left(x\right).
\end{aligned}
\end{equation}
Finally, note that for every \(E\)-invariant function \(\psi\in L^{\infty}\left(X,\mu\right)\) we have
\[S_{f\cdot\psi}^{E}\left(x\right)=S_{f}^{E}\left(x\right)\cdot\psi\left(x\right)\text{ for }\mu\text{-a.e. }x\in X,\]
so the identity \eqref{eq:Sf} when applied to \(f\cdot\psi\) implies the identity
\[\int_{X}f\left(x\right)\psi\left(x\right)d\mu\left(x\right)=\int_{X}S_{f}^{E}\left(x\right)S_{1}^{E}\left(x\right)^{-1}\psi\left(x\right)d\mu\left(x\right).\]
Then the \(E\)-invariance of \(S_{f}^{E}\left(x\right)S_{1}^{E}\left(x\right)^{-1}\) with the last identity readily imply that \(S_{f}^{E}\left(\cdot\right)S_{1}^{E}\left(\cdot\right)^{-1}=\mathbb{E}\left(f\mid\mathrm{Inv}\left(E\right)\right)\) in \(L^{1}\left(X,\mu\right)\).
\end{proof}

\section{Asymptotic Invariance of Hypercompact OERs}
\label{sct:asympinv}

The following asymptotic invariance property in hyperfinite equivalence relations was proved by Danilenko \cite[Lemma 2.2]{danilenko2001} and has been used by him for the random ratio ergodic theorem \cite[Theorem A.1]{danilenko2019}. We will show that Danilenko's proof can be adapted to hypercompact OERs as well.

Let \(G\) be an lcsc group with a left Haar measure \(\lambda\). A compact set \(S\subset G\) is said to be {\bf \(\left[K,\epsilon\right]\)-invariant}, for some compact set \(K\subset G\) and \(\epsilon>0\), if
\[\lambda\left(g\in S:Kg\subseteq S\right)>\left(1-\epsilon\right)\lambda\left(S\right).\]

\begin{prop}[Essentially Danilenko]
\label{prop:danil}
Let \(\left(X,\mu\right)\) be a probability preserving \(G\)-space. Suppose \(E_{G}^{X}\) is hypercompact with an equicompact filtration \(E_{1}\subseteq E_{2}\subseteq\dotsm\). Then for every compact set \(K\subset G\) and \(\epsilon>0\), it holds that
\[\liminf_{n\to\infty}\mu\left(x\in X:G_{E_{n}}\left(x\right)\text{is}\,\left[K,\epsilon\right]\text{-invariant}\right)>1-\epsilon.\]
\end{prop}

We first formulate a simple probability fact:

\begin{lem}
\label{lemp:prob}
Let \(W_{n}\xrightarrow[n\to\infty]{}W\) be an \(L^{1}\)-convergent sequence of random variables taking values in \(\left[0,1\right]\). Then for every \(0<\epsilon\leq 1/2\),
\[\text{if }\mathbb{E}\left(W\right)>1-\epsilon^{2}\text{ then }\liminf_{n\to\infty}\mathbb{P}\left(W_{n}>1-\epsilon\right)>1-\epsilon.\]
\end{lem}

\begin{proof}
For every \(n\), since \(W_{n}\) is taking values in \(\left[0,1\right]\),
\begin{align*}
\mathbb{E}\left(W_{n}\right)
&=\int_{\left\{ W_{n}>1-\epsilon\right\} }W_{n}d\mathbb{P}+\int_{\left\{ W_{n}\leq1-\epsilon\right\} }W_{n}d\mathbb{P}\\
&\leq\mathbb{P}\left(W_{n}>1-\epsilon\right)+\left(1-\epsilon\right)\mathbb{P}\left(W_{n}\leq1-\epsilon\right)\\
&=1-\epsilon\mathbb{P}\left(W_{n}\leq1-\epsilon\right).
\end{align*}
Since the left hand-side converges to \(\mathbb{E}\left(W\right)>1-\epsilon^{2}\) as \(n\to+\infty\), we deduce that \(\mathbb{P}\left(W_{n}\leq1-\epsilon\right)<\epsilon\) for every sufficiently large \(n\).
\end{proof}

\begin{proof}[Proof of Proposition \ref{prop:danil}]
For \(m\in\mathbb{N}\) let \(X_{m}\left(K\right)=\left\{ x\in X:K\subseteq G_{m}\left(x\right)\right\}\). From the equicompactness, \(X_{m}\left(K\right)\nearrow X\) as \(m\nearrow+\infty\), so pick \(m_{o}\in\mathbb{N}\) such that \(\mu\left(X_{m_{o}}\left(K\right)\right)>1-\epsilon^{2}\). By L\'{e}vy's martingale convergence theorem (see e.g. \cite[Chapter VII, \S4, Theorem 3 \& Problem 1]{shiryaev1996}),
\[W_{n}:=\mathbb{E}\left(1_{X_{m_{o}}\left(K\right)}\mid \mathrm{Inv}\left(E_{n}\right)\right)\xrightarrow[n\to+\infty]{L^{1}\left(\mu\right)}\mathbb{E}\left(1_{X_{m_{o}}\left(K\right)}\mid\mathrm{Inv}_{G}\left(X\right)\right):=W.\]
Since \(\mathbb{E}\left(W\right)=\mu\left(X_{m_{o}}\right)>1-\epsilon^{2}\), from Lemma \ref{lemp:prob} it follows that there exists \(n_{o}\in\mathbb{N}\), say \(n_{o}>m_{o}\), such that \(\mu\left(W_{n}>1-\epsilon\right)>1-\epsilon\) for every \(n>n_{o}\).

Abbreviate \(G_{n}\left(\cdot\right)=G_{E_{n}}\left(\cdot\right)\). For every \(n>n_{o}\), by the formula for conditional expectation as in Proposition \ref{prop:condexp} in the probability preserving case, since \(G_{m_{o}}\left(x\right)\subseteq G_{n}\left(x\right)\) and \(G_{n}\left(g.x\right)=G_{n}\left(x\right)g^{-1}\) for \(g\in G_{n}\left(x\right)\), we get
\begin{align*}
W_{n}\left(x\right)
&=\frac{\lambda\left(g\in G_{n}\left(x\right):K\subseteq G_{m_{o}}\left(g.x\right)\right)}{\lambda\left(G_{n}\left(x\right)\right)}\\
&\leq\frac{\lambda\left(g\in G_{n}\left(x\right):K\subseteq G_{n}\left(g.x\right)\right)}{\lambda\left(G_{n}\left(x\right)\right)}\\
&=\frac{\lambda\left(g\in G_{n}\left(x\right):Kg\subseteq G_{n}\left(x\right)\right)}{\lambda\left(G_{n}\left(x\right)\right)}
\end{align*}
Hence, for all \(n>n_{o}\),
\[\mu\left(x\in X:G_{n}\left(x\right)\text{is}\,\left[K,\epsilon\right]\text{-invariant}\right)\geq\mu\left(W_{n}>1-\epsilon\right)>1-\epsilon.\qedhere\]
\end{proof}

We can now prove Theorem \ref{thm:pmp}. First we mention that the setting of Theorem \ref{thm:pmp} is never void, since every lcsc group admits a probability preserving free \(G\)-space (see e.g. \cite[Remark 8.3]{avraham2024hopf}). Recall that an lcsc group \(G\) is {\bf amenable} if it admits a {\bf F\o lner sequence}, namely a compact filtration \(S_{1}\subseteq S_{2}\subseteq\dotsm\) of \(G\) such that for every compact set \(K\subset G\) and every \(\epsilon>0\), there is \(n_{o}\in\mathbb{N}\) such that \(S_{n}\) is \(\left[K,\epsilon\right]\)-invariant for all \(n>n_{o}\).

\begin{proof}[Proof of Theorem \ref{thm:pmp}]
One implication is a particular case of Theorem \ref{thm:CFW}. For the other implication, suppose that \(\left(X,\mu\right)\) is a probability preserving \(G\)-space and that, up to a \(\mu\)-null set, \(E_{G}^{X}\) admits an equicompact filtration \(E_{1}\subseteq E_{2}\subseteq\dotsm\). Fix an equicompact filtration \(K_{1}\subseteq K_{2}\subseteq\dotsm\) of \(G\), and a sequence \(\epsilon_{1},\epsilon_{2},\dotsm\) of positive numbers with \(\epsilon_{1}+\epsilon_{2}+\dotsm<+\infty\). From Proposition \ref{prop:danil} together with the Borel--Cantelli Lemma, one deduces that for \(\mu\)-a.e. \(x\in X\) and every \(m\in\mathbb{N}\), \(G_{E_{n}}\left(x\right)\) is \(\left[K_{m},\epsilon_{m}\right]\)-invariant for all but finitely many \(n\in\mathbb{N}\). It then follows that \(G_{1}\left(x\right)\subseteq G_{2}\left(x\right)\subseteq\dotsm\) is a F\o lner sequence of \(G\) for \(\mu\)-a.e. \(x\in X\), and in particular \(G\) is amenable.
\end{proof}

\section{Two Random Ratio Ergodic Theorems}
\label{sct:tworret}

The classical {\it Random Ergodic Theorem} of Kakutani regards the asymptotic behaviour of ergodic averages of the form
\[{\textstyle \frac{1}{n}\sum_{k=1}^{n}f\left(z_{k}.x\right),\quad n=1,2,\dotsc,}\]
where \(x\) is a point in a \(G\)-space, \(\left(z_{1},z_{2},\dotsc\right)\) is a random walk on \(G\), and \(f\) is an integrable function. See e.g. \cite[\S3]{furman2002} and the references therein. Performing a random walk on the acting group enables one the use of the classical ergodic averages along the natural F\o lner sequence of \(\mathbb{N}\). Alternatively, one may propose other methods to pick compactly many group elements at random in each step, and study the asymptotic behaviour of the corresponding ergodic averages. This was done by Danilenko for countable amenable groups \cite[Appendix A]{danilenko2019} using the Ornstein--Weiss Theorem about hyperfiniteness, and in the following we extend this method to uncountable amenable groups using Theorem \ref{thm:CFW} about hypercompactness.

\begin{thm}[First Random Ratio Ergodic Theorem]
\label{thm:1strret}
Let \(\left(X,\mu\right)\) be a nonsingular probability \(G\)-space and suppose that \(E_{G}^{X}\) is hypercompact with an equicompact filtration \(E_{1}\subseteq E_{2}\subseteq\dotsm\). Then for every \(f\in L^{1}\left(X,\mu\right)\),
\[\lim_{n\to\infty}\frac{\int_{G_{E_{n}}\left(x\right)}\nabla_{g}\left(x\right)f\left(g.x\right)d\lambda\left(g\right)}{\int_{G_{E_{n}}\left(x\right)}\nabla_{g}\left(x\right)d\lambda\left(g\right)}=\mathbb{E}\left(f\mid\mathrm{Inv}_{G}\left(X\right)\right)\left(x\right)\]
both \(\mu\)-a.e. and in \(L^{1}\left(X,\mu\right)\).
\end{thm}

It is worth noting that in the setting of Theorem \ref{thm:1strret}, while it is called an ergodic theorem, we do not say that the filtration has the F\o lner property.

\begin{proof}[Proof of Theorem \ref{thm:1strret}]
Since \(E_{1}\subseteq E_{2}\subseteq\dotsm\) is a filtration of \(E_{G}^{X}\) modulo \(\mu\), \(\mathrm{Inv}\left(E_{1}\right)\supset\mathrm{Inv}\left(E_{2}\right)\supset\dotsm\dotsc\) is an approximation of \(\mathrm{Inv}_{G}\left(X\right)\) modulo \(\mu\). Then by L\'{e}vy's martingale convergence theorem (see e.g. \cite[Chapter VII, \S4, Theorem 3 \& Problem 1]{shiryaev1996}), for every \(f\in L^{1}\left(X,\mu\right)\) it holds that
\[\lim_{n\to\infty}\mathbb{E}\left(f\mid\mathrm{Inv}\left(E_{n}\right)\right)=\mathbb{E}\left(f\mid\mathrm{Inv}_{G}\left(X\right)\right)\]
both \(\mu\)-a.e. and in \(L^{1}\left(X,\mu\right)\). By the formula for conditional expectation as in Proposition \ref{prop:condexp} this is the limit stated in the theorem.
\end{proof}

The following theorem was proved by Danilenko \cite[Appendix A]{danilenko2019} for countable amenable groups, using the Ornstein--Weiss Theorem. For the general case we will follow the same idea as Danilenko, only that we substitute the Ornstein--Weiss Theorem with Theorem \ref{thm:pmp} and use the formulas for conditional expectations as in Proposition \ref{prop:condexp}. It is worth mentioning that the general idea of the proof is the same, albeit considerably simpler, to the approach used by Bowen \& Nevo in establishing pointwise ergodic theorems for probability preserving actions of nonamenable (countable) groups. See e.g. \cite{bowen2013, bowen2015}.

\begin{defn}[Random Filtrations]
A {\bf random equicompact filtration} of \(G\) is a sequence \(\left(\mathcal{S}_{1},\mathcal{S}_{2},\dotsc\right)\) such that:
\begin{enumerate}
    \item Each \(\mathcal{S}_{n}\) is a function from an abstract probability space \(\left(\Omega,\mathbb{P}\right)\) into the relatively compact sets in \(G\).
    \item Each \(\mathcal{S}_{n}\) is measurable in the sense that \(\left\{ \left(\omega,g\right):g\in\mathcal{S}_{n}\left(\omega\right)\right\}\) is a measurable subset of \(\Omega\times G\).  
    \item \(\mathcal{S}_{1}\subseteq\mathcal{S}_{2}\subseteq\dotsm\) is \(\mathbb{P}\)-almost surely an equicompact filtration of \(G\).
\end{enumerate}
A {\bf random F\o lner sequence} is a random equicompact filtration \(\left(\mathcal{S}_{1},\mathcal{S}_{2},\dotsc\right)\) of \(G\) that forms \(\mathbb{P}\)-almost surely a F\o lner sequence of \(G\).
\end{defn}

Our source for Random F\o lner filtrations is, of course, hypercompact orbit equivalence relations:

\begin{lem}
\label{lem:randfil}
Suppose \(E_{G}^{X}\) is hypercompact with an equicompact filtration \(E_{1}\subseteq E_{2}\subseteq\dotsm\). Then the sequence \(\left(\mathcal{S}_{1},\mathcal{S}_{2},\dotsc\right)\) that is defined on \(\left(X,\mu\right)\) by 
\[\mathcal{S}_{n}\left(x\right)=G_{E_{n}}\left(x\right)=\left\{g\in G:\left(x,g.x\right)\in E\right\}\]
forms a random equicompact filtration.
\end{lem}

\begin{proof}
Of course, \(\left(\mathcal{S}_{1}\left(x\right),\mathcal{S}_{2}\left(x\right),\dotsc\right)\) forms an equicompact filtration of \(G\) for every \(x\in X\). In order to see the measurability, note that whenever \(E\subseteq E_{G}^{X}\) is an OER, and in particular a Borel subset of \(X\times X\), then the set \(\left\{\left(x,g\right)\in X\times G:g\in G_{E}\left(x\right)\right\}\) is nothing but the inverse image of \(E\) under the Borel map
\[X\times G\to X\times X,\quad\left(x,g\right)\mapsto \left(x,g.x\right).\qedhere\]
\end{proof}

Having all these notions defined we can prove now Theorem \ref{thm:rret}:

\begin{proof}[Proof of Theorem \ref{thm:rret}]
Pick any probability preserving free \(G\)-space \(\left(Y,\nu\right)\) (the existence of such a \(G\)-space is well-known; see e.g. \cite[Remark 8.3]{avraham2024hopf}). Since \(G\) is amenable, it follows from Theorem \ref{thm:CFW} that \(\left(Y,\nu\right)\) is hypercompact, and in fact there can be found an equicompact filtration \(F_{1}\subseteq F_{2}\subseteq\dotsm\) of \(E_{G}^{Y}\). Consider the diagonal nonsingular \(G\)-space \(\left(X\times Y,\mu\otimes\nu\right)\), namely with the action of \(G\) that is given by \(g.\left(x,y\right)=\left(g.x,g.y\right)\). Note that \(E_{G}^{X\times Y}\) is measure-hypercompact, since it has the equicompact filtration \(E_{1}\subseteq E_{2}\subseteq\dotsm\) given by
\[E_{n}=\left\{ \left(\left(x,y\right),\left(g.x,g.y\right)\right):\left(y,g.y\right)\in F_{n}\right\}.\]
Indeed, as the action of \(G\) on \(\left(Y,\nu\right)\) is free, it is easy to verify that
\[G_{n}\left(x,y\right):=G_{E_{n}}\left(x,y\right)=G_{F_{n}}\left(y\right),\quad\left(x,y\right)\in X\times Y.\]
Also, since \(G\) acts on \(\left(Y,\nu\right)\) in a probability preserving way, its Radon-Nikodym cocycle can be taken to be the constant function \(1\), so we have
\[\nabla_{g}^{\mu\otimes\nu}\left(x,y\right)=\nabla_{g}\left(x\right),\quad\left(x,y\right)\in X\times Y,\]
where \(\nabla_{g}^{\mu\otimes\nu}\) and \(\nabla_{g}\) denote the Radon--Nikodym cocycle \eqref{eq:rncoc} of the nonsingular \(G\)-spaces \(\left(X\times Y,\mu\otimes\nu\right)\) and \(\left(X,\mu\right)\), respectively. Then for every \(f\in L^{1}\left(X,\mu\right)\), applying Theorem \ref{thm:1strret} to \(f\otimes 1\in L^{1}\left(X\times Y,\mu\otimes\nu\right)\), we obtain
\[\lim_{n\to\infty}\frac{\int_{G_{n}\left(y\right)}\nabla_{g}\left(x\right)f\left(g.x\right)d\lambda\left(g\right)}{\int_{G_{n}\left(y\right)}\nabla_{g}\left(x\right)d\lambda\left(g\right)}=\mathbb{E}\big(f\otimes1\mid\mathrm{Inv}\big(E_{G}^{X\times Y}\big)\big)=\mathbb{E}\left(f\mid\mathrm{Inv}\left(E_{G}^{X}\right)\right)\]
both \(\mu\otimes\nu\)-a.e. and in \(L^{1}\left(X\times Y,\mu\otimes\nu\right)\). Define the random equicompact filtration \(\left(\mathcal{S}_{1},\mathcal{S}_{2},\dotsc\right)\) on the probability space \(\left(Y,\nu\right)\) by \(\mathcal{S}_{n}\left(y\right)=G_{E_{n}}\left(y\right)\), \(n\in\mathbb{N}\), which is a random equicompact filtration by Lemma \ref{lem:randfil}.

Finally, in order to see that \(\left(\mathcal{S}_{1},\mathcal{S}_{2},\dotsc\right)\) is a random F\o lner sequence, note that since \(F_{1}\subseteq F_{2}\subseteq\dotsm\) is an equicompact filtration of \(G\), then \(E_{1}\subseteq E_{2}\subseteq\dotsm\) is an equicompact filtration of \(E_{G}^{X\times Y}\). Then, as we have shown in the proof of Theorem \ref{thm:pmp}, the equicompactness ensures that for \(\nu\)-a.e. \(y\in Y\) this is a F\o lner sequence.
\end{proof}

\section{Hopf Dichotomy in Nonsingular Bernoulli Actions}
\label{sct:Hopf}

An important model in nonsingular ergodic theory is the {\bf nonsingular Bernoulli \(G\)-space}. For a countable group \(G\), this is the Borel \(G\)-space \(\left\{0,1\right\}^{G}\), where the action is given by translating the coordinates, and with a product measure of the form \(\bigotimes_{g\in G}\left(p_{g},1-p_{g}\right)\). A classical theorem of Kakutani provides a criterion on the summability of \(\left(p_{g}\right)_{g\in G}\) that completely determines when this action becomes nonsingular, in which case \(\big(\left\{0,1\right\}^{G},\bigotimes_{g\in G}\left(p_{g},1-p_{g}\right)\big)\) is called a {\bf nonsingular Bernoulli shift}. This model is well-studied in ergodic theory and its related fields in recent years, and many results are known about its conservativity, ergodicity and other ergodic-theoretical properties. See the survey \cite{danilenko2023ergodic}.

When it comes to an lcsc group \(G\), the natural model for nonsingular Bernoulli \(G\)-spaces is the nonsingular Poisson suspension.\footnote{Nonsingular Poisson suspension which has a dissipative base is referred to as nonsingular Bernoulli \(G\)-space, following the terminology in \cite[\S8]{danilenko2023krieger}.} We will introduced the fundamentals of this model below. The analog of the Kakutani dichotomy was established by Takahashi, and the formula for the Radon--Nikodym derivative was presented by Danilenko, Kosloff \& Roy \cite[Theorem 3.6]{danilenko2022nonsingular}. See also the survey \cite{danilenko2023ergodic}.

Here we will use the Random Ratio Ergodic Theorem \ref{thm:rret} in order to show that for a large class of nonsingular Bernoulli \(G\)-spaces the \emph{Hopf Dichotomy} holds: they are either totally dissipative or ergodic. Our method follows the main line of the \emph{Hopf method} in proving ergodicity, which was presented in the nonsingular category by Kosloff \cite{kosloff2019} and extended by Danilenko \cite{danilenko2019}.

\subsubsection{Conservativity and Dissipativity}

Let \(G\) be an lcsc group and \(\left(X,\mu\right)\) a nonsingular \emph{probability} \(G\)-space. Then \(\left(X,\mu\right)\) is said to be {\bf conservative} if it has the following recurrence property: for every Borel set \(A\subseteq X\) with \(\mu\left(A\right)>0\) and every compact set \(K\subset G\), there exists \(g\in G\backslash K\) such that \(\mu\left(A\cap g.A\right)>0\). It is a classical fact (for a proof see \cite{avraham2024hopf}) that conservativity is equivalent to that
\[\int_{G}\frac{d\mu\circ g}{d\mu}\left(x\right)d\lambda\left(g\right)=+\infty\text{ for }\mu\text{-a.e. }x\in X.\]
Accordingly, the \emph{Hopf Decomposition} or the \emph{Conservative--Dissipative Decomposition} of \(\left(X,\mu\right)\) is given by
\[\mathcal{D}:=\big\{ x\in X:\int_{G}\frac{d\mu\circ g}{d\mu}\left(x\right)d\lambda\left(g\right)<+\infty\big\} \text{ and }\mathcal{C}:=Y\backslash\mathcal{D}.\]
Thus, \(\left(X,\mu\right)\) is {\bf totally dissipative} if \(\mu\left(\mathcal{D}\right)=1\) and, on the other extreme, it is {\bf conservative} if \(\mu\left(\mathcal{C}\right)=1\).

\subsubsection{Nonsingular Bernoulli \(G\)-spaces}

Let \(G\) be an lcsc group and \(\mu\) be an absolutely continuous (i.e. in the Haar measure class) measure on \(G\). Denote the corresponding Radon--Nikodym derivative by
\[\partial:G\to\mathbb{R}_{>0},\quad \partial\left(x\right)=\frac{d\mu}{d\lambda}\left(x\right).\]
Thus, \(\left(G,\mu\right)\) is a nonsingular \(G\)-space in the action of \(G\) on itself by translations, and we have the Radon--Nikodym cocycle
\[\nabla_{g}\left(x\right):=\frac{d\mu\circ g}{d\mu}\left(x\right)=\frac{d\mu}{d\lambda}\left(gx\right)\frac{d\lambda}{d\mu}\left(x\right)=\frac{\partial\left(gx\right)}{\partial\left(x\right)},\quad g,x\in G.\]
Here \(\mu\circ g\) is the measure \(A\mapsto\mu\left(gA\right)\) (\(g\) forms an invertible transformation).

Let \(G^{\ast}\) be the space of counting Radon measures on \(G\). Thus an element \(p\in G^{\ast}\) is a Radon measure on \(G\) that takes nonnegative integer values. The Borel \(\sigma\)-algebra of \(G^{\ast}\) is generated by the mappings
\[N_{A}:G^{\ast}\to\mathbb{Z}_{\geq 0}\cup\left\{+\infty\right\},\quad N_{A}\left(p\right)=p\left(A\right),\]
for all Borel sets \(A\subseteq G\). Denote by \(\mu^{\ast}\) the unique probability measure on \(G^{\ast}\) with respect to which the distribution of each of the random variables \(N_{A}\) is Poisson with mean \(\mu\left(A\right)\), meaning that
\[\mu^{\ast}\left(N_{A}=k\right)=\frac{e^{-\mu\left(A\right)}\mu\left(A\right)^{k}}{k!},\quad k\in\mathbb{Z}_{\geq0}.\]
It is a basic fact that when \(G\) is non-discrete so that \(\mu\) is nonatomic, for every \(p\) in a \(\mu^{\ast}\)-conull set in \(G^{\ast}\) there are no repetitions, that is to say, every such \(p\) is of the form \(p=\sum_{p\in P}\delta_{p}\), where \(P\subset G\) is some discrete set that we will denote \(P=\mathrm{Supp}\left(p\right)\) and refer to as the {\bf support} of \(p\).

There is a natural Borel action of \(G\) on \(G^{\ast}\), which is given by
\[g.p\left(A\right)=p\left(gA\right),\quad g\in G.\]
By \cite[Theorem 3.6]{danilenko2022nonsingular} (see also \cite[Theorem 3.1]{danilenko2023ergodic}), when \(\mu\) satisfies
\[\nabla_{g}-1\in L^{1}\left(G,\mu\right)\text{ for all }g\in G,\]
or, equivalently,
\[g.\partial-\partial\in L^{1}\left(G,\lambda\right)\text{ for all }g\in G,\]
then the action of \(G\) on \(\left(G^{\ast},\mu^{\ast}\right)\) is nonsingular, and the Radon--Nikodym derivatives have the formula
\begin{equation}
\begin{aligned}
\label{eq:RNform}
\nabla_{g}^{\ast}\left(p\right):=\frac{d\mu^{\ast}\circ g}{d\mu^{\ast}}\left(p\right)
&=e^{-\int_{G}\left(\nabla_{g}-1\right)d\mu}\cdot\prod\nolimits_{x\in\mathrm{Supp}\left(p\right)}\nabla_{g}\left(x\right) \\
&=e^{-\int_{G}\left(g.\partial-\partial\right)d\lambda}\cdot\prod\nolimits_{x\in\mathrm{Supp}\left(p\right)}\frac{\partial\left(gx\right)}{\partial\left(x\right)},\quad g\in G.
\end{aligned}
\end{equation}

\subsubsection{A Hopf Dichotomy}

The following theorem is a general result about nonsingular Bernoulli \(G\)-spaces. For constructions of nonsingular Bernoulli \(G\)-spaces with a variety of ergodic properties we refer to \cite{danilenko2022generic, danilenko2023krieger, danilenko2022krieger}.

\begin{thm}
\label{thm:conserg}
Let \(G\) be an lcsc \emph{amenable} group and \(\mu\) an absolutely continuous measure on \(G\) such that \(\partial:=d\mu/d\lambda\) satisfies
\[0<\inf\partial<\sup\partial<+\infty\text{ and }g.\partial-\partial\in L^{1}\left(G,\lambda\right),\quad g\in G.\]
Then the nonsingular \(G\)-space \(\left(G^{\ast},\mu^{\ast}\right)\) is either totally dissipative or ergodic.
\end{thm}

The first part of the proof, namely that \(\left(G^{\ast},\mu^{\ast}\right)\) is either totally dissipative or conservative, is essentially the same as the proof of \cite[Theorem 8.2]{danilenko2022generic}. The main part of the proof, namely that conservativity implies ergodicity, uses Corollary \ref{cor:danilrret} of the Random Ratio Ergodic Theorem. Both parts relies on the following observation of Danilenko, Kosloff \& Roy.

Denote by \(\Pi=\Pi\left(G,\mu\right)\) the group of all \(\mu\)-preserving invertible compactly supported transformations of \(\left(G,\mu\right)\). We denote by \(\mathrm{Supp}\left(\pi\right)\) the (compact) support of \(\pi\in\Pi\). Then \(\Pi\) acts naturally on \(\left(G^{\ast},\mu^{\ast}\right)\) in a measure preserving way via
\[\pi.p\left(A\right)=p\left(\pi^{-1}\left(A\right)\right),\quad \pi\in\Pi,\,p\in G^{\ast}.\]
As it was proved in the course of the proof \cite[Theorem 8.2]{danilenko2022generic}, we have:

\begin{fct}
The action of \(\Pi=\Pi\left(G,\mu\right)\) on \(\left(G^{\ast},\mu^{\ast}\right)\) is ergodic.
\end{fct}

\begin{proof}[Proof of Theorem \ref{thm:conserg}]
First, note that for every \(g\in G\) we have
\[\int_{G}\left|g.\partial/\partial-1\right|d\mu=\int_{G}\left|g.\partial-\partial\right|d\lambda.\]
Then since \(g.\partial-\partial\in L^{1}\left(G,\lambda\right)\) for every \(g\in G\), it follows from \cite[Theorem 3.6]{danilenko2022nonsingular}) that \(\left(G^{\ast},\mu^{\ast}\right)\) is a nonsingular \(G\)-space.

\subsubsection*{Part 1}

Look at the dissipative part of \(\left(G^{\ast},\mu^{\ast}\right)\), that we denote
\[\mathcal{D}^{\ast}:=\big\{ p\in G^{\ast}:\int_{G}\nabla_{g}^{\ast}\left(p\right)d\lambda\left(g\right)<+\infty\big\}.\]
Let us show that \(\mathcal{D}^{\ast}\) is \(\Pi\)-invariant. Find \(0<\alpha<1\) with \(\alpha\leq\inf\partial<\sup\partial\leq\alpha^{-1}\). For every \(\pi\in\Pi\) consider the function
\[\upsilon_{\pi}:G^{\ast}\to\mathbb{R}_{>0},\quad\upsilon_{\pi}\left(p\right)=\alpha^{\#\left(\mathrm{Supp}\left(p\right)\cap \mathrm{Supp}\left(\pi\right)\right)}=\alpha^{N_{\mathrm{Supp}\left(\pi\right)}\left(p\right)}.\]
Since \(\mathrm{Supp}\left(\pi\right)\) is compact, for \(\mu^{\ast}\)-a.e. \(p\in G^{\ast}\) we have \(N_{\mathrm{Supp}\left(\pi\right)}\left(p\right)<+\infty\). By the formula \eqref{eq:RNform} we deduce
\begin{equation}
\small
\label{eq:epsilon}
\frac{\nabla_{g}^{\ast}\left(\pi.p\right)}{\nabla_{g}^{\ast}\left(p\right)} =
\prod\nolimits_{x\in\mathrm{Supp}\left(p\right)\cap \mathrm{Supp}\left(\pi\right)}\frac{\nabla_{g}\left(\pi\left(x\right)\right)}{\nabla_{g}\left(x\right)}
\in \big[\upsilon_{\pi}\left(p\right)^{4}, \upsilon_{\pi}\left(p\right)^{-4}\big].
\end{equation}
This readily implies the \(\Pi\)-invariance of \(\mathcal{D}^{\ast}\), and since \(\Pi\) acts ergodically we deduce that \(\left(G^{\ast},\mu^{\ast}\right)\) is either totally dissipative or conservative.

\subsubsection*{Part 2}

We use Corollary \ref{cor:danilrret} of the Random Ratio Ergodic Theorem in order to establish the following property: Assuming that \(\left(G^{\ast},\mu^{\ast}\right)\) is conservative, then for every \(\pi\in\Pi\) there exists a function \(\Upsilon_{\pi}:G^{\ast}\to\mathbb{R}_{>0}\) such that for every nonnegative function \(\varphi\in L^{1}\left(G^{\ast},\mu^{\ast}\right)\) and \(\mu^{\ast}\)-a.e. \(p\in G^{\ast}\),
\begin{equation}
\label{eq:asymp}\tag{\({\scriptstyle \clubsuit}\)}
\small
\Upsilon_{\pi}\left(p\right)\cdot\mathbb{E}\left(\varphi\mid\mathrm{Inv}_{G}\left(G^{\ast}\right)\right)\left(p\right)\leq\mathbb{E}\left(\varphi\mid\mathrm{Inv}_{G}\left(G^{\ast}\right)\right)\left(\pi.p\right)\leq\frac{\mathbb{E}\left(\varphi\mid\mathrm{Inv}_{G}\left(G^{\ast}\right)\right)\left(p\right)}{\Upsilon_{\pi}\left(p\right)}
\end{equation}
In fact, we will show that \(\Upsilon_{\pi}\left(p\right):=\upsilon_{\pi}\left(p\right)^{16}=\alpha^{16\cdot N_{\mathrm{Supp}\left(\pi\right)}\left(p\right)}\) suffices.

Let \(\varphi\in L^{1}\left(G^{\ast},\mu^{\ast}\right)\) be some nonnegative function. By a standard approximation argument, in order to establish \eqref{eq:asymp} it is sufficient to assume that \(\varphi\) is uniformly continuous in the vague topology on \(G^{\ast}\), which induces its Borel structure. For such \(\varphi\), for every \(\pi\in\Pi\) one can easily verify that for all \(f\in C_{0}\left(G\right)\) (i.e. continuous and vanishes as \(g\to\infty\)),
\[\int_{G}f\left(g\pi\left(x\right)\right)dp\left(x\right)-\int_{G}f\left(gx\right)dp\left(x\right)\to0\text{ as }g\to\infty.\]
That is to say, \(d\left(g.\pi.p,g.p\right)\to0\) as \(g\to\infty\), where \(d\) denotes the metric of the vague topology on \(G^{\ast}\). Then by the uniform continuity of \(\varphi\) we obtain
\begin{equation}
\label{eq:asymp1}
\varphi\left(g.\pi.p\right)-\varphi\left(g.p\right)\to0\text{ as }g\to\infty\text{ uniformly in }p.
\end{equation}
Using Corollary \ref{cor:danilrret}, there exists a F\o lner sequence \(S_{1}\subseteq S_{2}\subseteq\dotsm\) of \(G\), corresponding to \(\mathcal{L}=\left\{\varphi\right\}\), whose ergodic averages
\[A_{n}\varphi\left(p\right):=\frac{\int_{S_{n}}\nabla_{g}^{\ast}\left(p\right)\varphi\left(g.p\right)d\lambda\left(g\right)}{\int_{S_{n}}\nabla_{g}^{\ast}\left(p\right)d\lambda\left(g\right)},\quad p\in G^{\ast},\,n\in\mathbb{N},\]
satisfy that
\[\lim_{n\to\infty}A_{n}\varphi\left(p\right)=\mathbb{E}\left(\varphi\mid\mathrm{Inv}_{G}\left(G^{\ast}\right)\right)\left(p\right)\text{ for }\mu^{\ast}\text{-a.e. }p\in G^{\ast}.\]
From \eqref{eq:epsilon}, for \(\mu^{\ast}\)-a.e. \(p\in G^{\ast}\) and every \(n\in\mathbb{N}\) we have
\begin{equation}
\label{eq:EA1}
\scalebox{0.9}{\(\upsilon_{\pi}\left(p\right)^{16}\cdot\frac{\int_{S_{n}}\nabla_{g}^{\ast}\left(p\right)\varphi\left(g.\pi.p\right)d\lambda\left(g\right)}{\int_{S_{n}}\nabla_{g}^{\ast}\left(p\right)d\lambda\left(g\right)}\leq A_{n}\varphi\left(\pi.p\right)\leq\upsilon_{\pi}\left(p\right)^{-16}\cdot\frac{\int_{S_{n}}\nabla_{g}^{\ast}\left(p\right)\varphi\left(g.\pi.p\right)d\lambda\left(g\right)}{\int_{S_{n}}\nabla_{g}^{\ast}\left(p\right)d\lambda\left(g\right)}\).}
\end{equation}
Recall that by the conservativity of the action of \(G\) on \(\left(G^{\ast},\mu^{\ast}\right)\) we have that
\[\lim_{n\to+\infty}\int_{S_{n}}\nabla_{g}^{\ast}\left(p\right)d\lambda\left(g\right)=\int_{G}\nabla_{g}^{\ast}\left(p\right)d\lambda\left(g\right)=+\infty,\]
so using \eqref{eq:asymp1} we deduce that
\begin{equation}
\label{eq:EA2}
\small
\begin{aligned}
& \lim_{n\to\infty}\left|\frac{\int_{S_{n}}\nabla_{g}^{\ast}\left(p\right)\varphi\left(g.\pi.p\right)d\lambda\left(g\right)}{\int_{S_{n}}\nabla_{g}^{\ast}\left(p\right)d\lambda\left(g\right)}-A_{n}\varphi\left(p\right)\right|\\
& \qquad\qquad\leq \lim_{n\to\infty}\frac{\int_{S_{n}}\nabla_{g}^{\ast}\left(p\right)\left|\varphi\left(g.\pi.p\right)-\varphi\left(g.p\right)\right|d\lambda\left(g\right)}{\int_{S_{n}}\nabla_{g}^{\ast}\left(p\right)d\lambda\left(g\right)}=0.
\end{aligned}
\end{equation}
Finally, when taking the limit as \(n\to+\infty\) in \eqref{eq:EA1}, together with \eqref{eq:EA2} we obtain \eqref{eq:asymp} for the function \(\Upsilon_{\pi}\left(p\right)=\upsilon_{\pi}\left(p\right)^{16}\).

\subsubsection*{Part 3}

We finally deduce that if \(\left(G^{\ast},\mu^{\ast}\right)\) is conservative then it is ergodic. If \(E\subseteq G^{\ast}\) is a \(G\)-invariant set with \(\mu^{\ast}\left(E\right)>0\), from \eqref{eq:asymp} for the indicator \(\varphi=1_{E}\) we obtain that \(1_{E}\left(\pi.p\right)=1_{E}\left(p\right)\) for every \(\pi\in\Pi\) and \(\mu^{\ast}\)-a.e. \(p\in G^{\ast}\), thus \(E\) is a \(\Pi\)-invariant. Since \(\Pi\) acts ergodically we deduce \(\mu^{\ast}\left(E\right)=1\).
\end{proof}

\subsection*{Acknowledgments}

We would like to express our gratitude to Michael Bj\"{o}rklund for many fruitful discussions regarding this work, and to Zemer Kosloff for suggesting that an uncountable version of Danilenko's Random Ratio Ergodic Theorem should be true, and for providing us with useful comments and corrections. We are also grateful to Alexander Kechris who, after this work was first published, communicated to us Proposition \ref{prop:kechris}.

\bibliographystyle{acm}
\bibliography{References}

\begin{thebibliography}{10}

\bibitem{avraham2024hopf}
{\sc Avraham-Re'em, N., and Peterzil, G.}
\newblock The {H}opf decomposition.
\newblock {\em arXiv preprint arXiv:2406.17137\/} (2024).

\bibitem{bowen2013}
{\sc Bowen, L., and Nevo, A.}
\newblock Pointwise ergodic theorems beyond amenable groups.
\newblock {\em Ergodic Theory and Dynamical Systems 33}, 3 (2013), 777--820.

\bibitem{bowen2015}
{\sc Bowen, L., and Nevo, A.}
\newblock Amenable equivalence relations and the construction of ergodic averages for group actions.
\newblock {\em Journal d'Analyse Math{\'e}matique 126}, 1 (2015), 359--388.

\bibitem{connes1981}
{\sc Connes, A., Feldman, J., and Weiss, B.}
\newblock An amenable equivalence relation is generated by a single transformation.
\newblock {\em Ergodic theory and dynamical systems 1}, 4 (1981), 431--450.

\bibitem{danilenko2022krieger}
{\sc Danilenko, A., and Kosloff, Z.}
\newblock Krieger’s type of nonsingular {P}oisson suspensions and {IDPFT} systems.
\newblock {\em Proceedings of the American Mathematical Society 150}, 4 (2022), 1541--1557.

\bibitem{danilenko2001}
{\sc Danilenko, A.~I.}
\newblock Entropy theory from the orbital point of view.
\newblock {\em Monatshefte f{\"u}r Mathematik 134\/} (2001), 121--141.

\bibitem{danilenko2019}
{\sc Danilenko, A.~I.}
\newblock Weak mixing for nonsingular {B}ernoulli actions of countable amenable groups.
\newblock {\em Proceedings of the American Mathematical Society 147}, 10 (2019), 4439--4450.

\bibitem{danilenko2023krieger}
{\sc Danilenko, A.~I.}
\newblock Krieger’s type for ergodic non-singular {P}oisson actions of non-({T}) locally compact groups.
\newblock {\em Ergodic Theory and Dynamical Systems 43}, 7 (2023), 2317--2353.

\bibitem{danilenko2022generic}
{\sc Danilenko, A.~I., Kosloff, Z., and Roy, E.}
\newblock Generic non-singular {P}oisson suspension is of type $\mathrm{III}_1$.
\newblock {\em Ergodic Theory and Dynamical Systems 42}, 4 (2022), 1415--1445.

\bibitem{danilenko2022nonsingular}
{\sc Danilenko, A.~I., Kosloff, Z., and Roy, E.}
\newblock Nonsingular {P}oisson suspensions.
\newblock {\em Journal d'Analyse Math{\'e}matique 146}, 2 (2022), 741--790.

\bibitem{danilenko2023ergodic}
{\sc Danilenko, A.~I., and Silva, C.~E.}
\newblock Ergodic theory: {N}onsingular transformations.
\newblock In {\em Mathematics of complexity and dynamical systems. {V}ols. 1--3}. Springer, New York, 2023, pp.~329--356.

\bibitem{dougherty1994}
{\sc Dougherty, R., Jackson, S., and Kechris, A.~S.}
\newblock The structure of hyperfinite {B}orel equivalence relations.
\newblock {\em Transactions of the American mathematical society 341}, 1 (1994), 193--225.

\bibitem{furman2002}
{\sc Furman, A.}
\newblock Random walks on groups and random transformations.
\newblock In {\em Handbook of dynamical systems}, vol.~1. Elsevier, 2002, pp.~931--1014.

\bibitem{gao2008}
{\sc Gao, S.}
\newblock {\em Invariant Descriptive Set Theory}.
\newblock CRC Press, 2008.

\bibitem{hochman2013}
{\sc Hochman, M.}
\newblock On the ratio ergodic theorem for group actions.
\newblock {\em Journal of the London Mathematical Society 88}, 2 (2013), 465--482.

\bibitem{jackson2002}
{\sc Jackson, S., Kechris, A.~S., and Louveau, A.}
\newblock Countable {B}orel equivalence relations.
\newblock {\em Journal of mathematical logic 2}, 01 (2002), 1--80.

\bibitem{kechris1992}
{\sc Kechris, A.~S.}
\newblock Countable sections for locally compact group actions.
\newblock {\em Ergodic theory and dynamical systems 12}, 2 (1992), 283--295.

\bibitem{kechris2012}
{\sc Kechris, A.~S.}
\newblock {\em Classical Descriptive Set Theory}.
\newblock Springer, 2012.

\bibitem{kechris2024}
{\sc Kechris, A.~S.}
\newblock {\em The {T}heory of {C}ountable {B}orel {E}quivalence {R}elations}.
\newblock Cambridge Tracts in Mathematics 234. Cambridge University Press, 2024.

\bibitem{kosloff2019}
{\sc Kosloff, Z.}
\newblock Proving ergodicity via divergence of time averages.
\newblock {\em Studia Mathematica 248\/} (2019), 191--215.

\bibitem{ornstein1980}
{\sc Ornstein, D.~S., and Weiss, B.}
\newblock Ergodic theory of amenable group actions. {I}. {T}he {R}ohlin lemma.
\newblock {\em Bulletin of the American Mathematical Society 2}, 1 (1980), 161--164.

\bibitem{shiryaev1996}
{\sc Shiryaev, A.~N.}
\newblock {\em Probability}, 2~ed., vol.~95 of {\em Graduate Texts in Mathematics}.
\newblock Springer, 1996.

\bibitem{slutsky2017}
{\sc Slutsky, K.}
\newblock Lebesgue orbit equivalence of multidimensional {B}orel flows: a picturebook of tilings.
\newblock {\em Ergodic Theory and Dynamical Systems 37}, 6 (2017), 1966--1996.

\bibitem{struble1974}
{\sc Struble, R.~A.}
\newblock Metrics in locally compact groups.
\newblock {\em Compositio Mathematica 28}, 3 (1974), 217--222.

\bibitem{varadarajan1968}
{\sc Varadarajan, V.~S.}
\newblock {\em Geometry of quantum theory}, vol.~1.
\newblock Springer, 1968.

\bibitem{zimmer2013}
{\sc Zimmer, R.~J.}
\newblock {\em Ergodic theory and semisimple groups}, vol.~81.
\newblock Springer Science \& Business Media, 2013.

\end{thebibliography}

\nocite{*}

\end{document}